\documentclass[11pt, reqno]{amsart}
\usepackage{amsfonts,amssymb,latexsym,amsmath, amsxtra, nccmath, amsthm}
\usepackage{verbatim}
\usepackage{url}
\usepackage{enumitem}
\usepackage{tikz}
\usepackage{cancel,comment}
\usepackage{bm}
\usepackage{csquotes}
\usepackage{booktabs}
\usepackage{longtable}
\allowdisplaybreaks
\theoremstyle{plain}
\newtheorem{theorem}{Theorem}[section]
\newtheorem*{theorem*}{Theorem}

\newtheorem{lemma}[theorem]{Lemma}
\newtheorem{proposition}[theorem]{Proposition}
\newtheorem*{conjecture*}{Conjecture}

\theoremstyle{definition}

\theoremstyle{remark}
\newtheorem*{remark}{Remark}
\newtheorem*{remark*}{Remark}
\newtheorem*{remarks*}{Remarks}

\numberwithin{equation}{section}

\newcommand{\R}{\mathbb R}
\newcommand{\N}{\mathbb N}
\newcommand{\Z}{\mathbb Z}
\newcommand{\C}{\mathbb C}

\newcommand{\Q}{{\mathbb Q}}

\newcommand{\lcm}{\operatorname{lcm}}

\newif\ifextra
\newif\ifdiscussion

\extrafalse
\discussionfalse

\def\({\left(}
\def\){\right)}

\def\GG{\mathcal{G}}

\renewcommand{\pmod}[1]{\,\,\left({\rm mod}\,\,{#1}\right)}

\newcommand{\im}{\operatorname{Im}}

\newcommand{\ord}{\operatorname{ord}}

\newcommand{\SL}{\operatorname{SL}}

\renewcommand{\H}{\mathbb{H}}

\newcommand{\bma}{\bm{a}}
\newcommand{\aj}{a}
\newcommand{\vol}{\operatorname{vol}}

\newenvironment{extradetailsproof}{\ifextra \noindent  \newline \noindent \bf
************************** Begin Extra Details **************************
\rm \fi } {\ifextra \noindent 
\noindent \bf ***************************  End Extra Details  ***************************\rm\\ \fi }

\ifextra
\else
\excludecomment{extradetails}
\excludecomment{extradetailsproof}
\fi

\ifdiscussion
\else
\excludecomment{discussion}
\fi

\title{Universal sums of generalized heptagonal numbers}
\author{Ramanujam Kamaraj}
\address{Department of Mathematics\\ University of Hong Kong\\ Pokfulam, Hong Kong}
\email{ramu72@connect.hku.hk}
\author{Ben Kane}
\address{Department of Mathematics\\ University of Hong Kong\\ Pokfulam, Hong Kong}
\email{bkane@hku.hk}
\date{\today}
\author{Ryoko Tomiyasu}
\address{Institute of Mathematics for Industry\\ Kyushu University\\ Kyushu, Japan}
\email{tomiyasu.ryoko.446@m.kyushu-u.ac.jp}
\thanks{The research of the second author was supported by a grant from the Research Grants Council of the Hong Kong SAR, China (project numbers HKU 17303618). This project was financially supported by the SENTAN-Q program of MEXT Initiative for Realizing Diversity in the Research Environment, JSPS KAKENHI (19K03628).}
\begin{document}
\begin{abstract}
In this paper, we consider representations of integers as sums of heptagonal numbers with a prescribed number of repeats of each heptagonal number appearing in the sum. In particular, we investigate the classification of such sums which are universal, i.e., those that represent every positive integer. We prove an explicit finite bound such that a given sum is universal if and only if it represents positive integer up to the given bound.
\end{abstract}

\maketitle

\section{Introduction and statement of results}

For $m\in\N$ with $m\geq 3$ and $x\in\Z$, we let 
\[
p_m(x):=\frac{(m-2)x^2-(m-4)x}{2}
\]
denote the \begin{it}$x$-th generalized $m$-gonal number\end{it} (if $x\in\N_0:=\N\cup\{0\}$, then we simply call this the \begin{it}$x$-th $m$-gonal number\end{it}, which counts the number of dots needed to form the shape of a regular $m$-gon of length $x$). Representations of integers as sums of polygonal numbers has a long history. For example, Fermat conjectured in 1638 that every positive integer may be written as the sum of at most $m$ $m$-gonal numbers. This was proven in a series of famous papers; the case $m=4$ is Lagrange's four squares theorem, proven in 1770, Gauss resolved the $m=3$ case in his 1796 ``Eureka Theorem'', and the general result was proven by Cauchy \cite{Cauchy} in 1813. More generally, for $\bma\in\N^{\ell}$ and $n\in\N$, we are interested in studying whether $n$ may be written as a sum of $m$-gonal numbers of the type 
\begin{equation}\label{eqn:sumsmgonal}
n=\sum_{j=1}^{\ell} \aj_j p_{m}(x_j).
\end{equation}
Choices of $\bma$ for which \eqref{eqn:sumsmgonal} is solvable for every positive integer, such as those considered by Fermat, are called \begin{it}universal\end{it} sums of (generalized) polygonal numbers. In the case of squares (i.e., $m=4$), Ramanujan \cite{RamanujanQuaternary} determined the $54$ choices of quaternary (i.e., $\ell=4$) $\bma$ giving universal forms, as later formally proven by Dickson \cite{Dickson}. Willerding \cite{Willerding} and others studied universality more generally for quadratic forms, culminating in the fifteen theorem of Conway--Schneeberger \cite{Conway,Bhargava}, which states that a classical quadratic form (sums of squares are classical) is universal if and only if it represents every positive integer up to $15$. Generalizing this for other choices of $m$, it was shown in \cite{KL} that for every $m$ there exists (an optimal) $\gamma_m\in\N$ such that \eqref{eqn:sumsmgonal} is universal if and only if it represents every integer up to $\gamma_m$. It was furthermore proven in \cite{KL} that $m\ll \gamma_m\ll_{\varepsilon} m^{7+\varepsilon}$ as $m\to\infty$ and this bound was later reduced by Kim and Park \cite{KimPark} to show that the growth of $\gamma_m$ is indeed linear in $m$ (i.e., $m\ll \gamma_m\ll m$). It is unclear whether the limit $\lim_{m\to\infty} \frac{\gamma_m}{m}$ exists or not.

Going in another direction, instead of studying the asymptotic growth as $m\to\infty$, some explicit values of $\gamma_m$ have been determined for small fixed choices of $m$. Specifically, it has been shown by Bosma and the second author in \cite{BosmaKane} that $\gamma_6=\gamma_3=8$, the aforementioned Conway--Schneeberger 15 theorem implies that $\gamma_4=15$, Ju proved that $\gamma_5=109$ \cite{Ju}, and Ju and Oh showed that $\gamma_8=60$ \cite{JuOh}. The main result in this paper is a bound for $\gamma_7$.
\begin{theorem}\label{thm:gamma7bound}
If \eqref{eqn:sumsmgonal} is solvable with $\bm{x}\in\Z^{\ell}$ for every positive integer $n\leq 3.896\cdot 10^{106}$, then the sum of generalized polygonal numbers $\sum_{j=1}^{\ell}\aj_jp_m(x_j)$ is universal.

In particular, we have $\gamma_7\leq 3.896\cdot 10^{106}$.
\end{theorem}
\begin{remark}
A numerical calculation indicates that $\gamma_7=131$. Given the bounds in Theorem \ref{thm:gamma7bound}, it turns out that to verify that $\gamma_7=131$ one only needs to find solutions to \eqref{eqn:sumsmgonal} for finitely-many $n$ and finitely many choices of $\bma$. Unfortunately, the bounds in Theorem \ref{thm:gamma7bound} are too large to feasibly check numerically. We note, however, that universality of some choices of $\bma$ follow from a result of Legendre, who showed in 1830 that for $m$ odd every $n\geq 28(m-2)^3$ is a sum of four $m$-gonal numbers, and a recent extension by Kim \cite{Kim} that generalizes this for any $\bma$ containing only $1$s and $2$s. Proceeding in this manner would require a number of clever ad-hoc arguments due to the need to generalize Legendre's result for many different cases.
\end{remark}
Our proof of Theorem \ref{thm:gamma7bound} goes through the analytic theory of quadratic forms and modular forms. We consider the theta function (for $q:=e^{2\pi i\tau}$)
\[
\Theta_{m,\bma}(\tau):=\sum_{n\geq 0} r_{m,\bma}^*(n) q^n,
\]
where 
\[
r_{m,\bma}^*(n):=\#\left\{\bm{x}\in\Z^{\ell}: \sum_{1\leq j\leq d} \aj_j p_{m}(x_j)=n\right\}.
\]
Since universality of $\sum_{j=1}^{\ell}a_j p_m(x_j)$ is equivalent to showing that $r_{m,\bma}^*(n)>0$ for every $n\in\N$. one decomposes $r_{m,\bma}^*(n)$ into a ``main term'' coming from coefficients of an Eisenstein series and an ``error term'' coming from a cusp form (see \eqref{eqn:decomposition} and \eqref{eqn:EisCuspSplit}). Since the space of modular forms is finite-dimensional, in principle one can obtain the decomposition explicitly in terms of certain basis elements by doing appropriate linear algebra and then find lower and upper bounds for (the absolute value of) the coefficients of the main term and error term, using a fundamental theorem of Deligne \cite{Deligne} to bound the coefficients of the cusp form. In practice, this linear algebra, done over an appropriate number field above $\Q$, is rather time-consuming and computationally expensive, so we side-step this issue by obtaining uniform lower bounds for the coefficients of the Eisenstein series and upper bounds for the absolute value of the coefficients of the cuspidal part by employing some metrics related to the quadratic polynomials from which the question arose, particularly, we obtain the bounds in terms of $m$ and $\bma$.

 The Eisentein series projection arises in a number of advantageous forms. Firstly, it may be realized (see \eqref{eqn:SiegelWeil}) as an average over the theta series coming from a ``genus'' of quadratic polynomials similar to $\sum_{j=1}^{\ell} \aj_j p_m(x_j)$, which, roughly speaking, is the set of quadratic polynomials which are indistinguishable from the sum of polygonal numbers over the adeles. Equivalently, they are indistinguishable $p$-adically for every prime $p$ and over the real numbers.   Secondly, using the local information, one may realize the coefficients of the Eisenstein series as certain products of ``local densities'' (see \eqref{eqn:localdensity}); recognizing this product of local densities as the $L$-function associated to a character up to finitely-many exceptional primes, one obtains a formula (see \cite{SiegelQuadForm,Weil,vanderBlij,ShimuraCongruence} for explicit formulas) that can be used to obtain a lower bound for the main term (see Section \ref{sec:Eisenstein}).

The cuspidal part is dealt with in Section \ref{sec:cuspidal}. We use a generalization of Banerjee and the second author \cite{BanerjeeKane} of a method employed by both Blomer \cite{Blomer} and Duke \cite{DukeTernary} to obtain bounds for coefficients of the cuspidal part in terms of the Petersson norm of the cusp form and certain natural metrics on the quadratic form (generalized to certain quadratic polynomials and with the constants explicitly computed in \cite{BanerjeeKane}). Although a bound on the norm in terms of $m$ was computed in \cite{KL}, the result is useless for our purposes, as we are fixing $m=7$. In order to compare the upper bound of the absolute value of the cuspidal part with the lower bound of the coefficients of the Eisenstein series part, we require explicit constants, and one does not easily recover the constants from the calculation in \cite{KL} (which is based on Duke's \cite{DukeTernary} calculation for quadratic forms); instead, we simplify the calculation by emulating Blomer's \cite{Blomer} calculation instead of Duke's. The resulting bound may be found in Lemma \ref{lem:cuspbound}.

The paper is organized as follows. In Section \ref{sec:prelim}, we give some preliminaries and background information. In Section \ref{sec:cuspidal}, we prove an upper bound on the Petersson norm (and hence the coefficients) of the cuspidal part in Lemma \ref{lem:cuspbound}. We then prove a lower bound on the coefficients of the Eisenstein series part in Section \ref{sec:Eisenstein}, and finally prove Theorem \ref{thm:gamma7bound} in Section \ref{sec:mainproof}.

\section{Some preliminaries}\label{sec:prelim}

\subsection{Congruence subgroups}
We require some well-known identities and relations involving congruence subgroups of $\SL_2(\Z)$. We first define the congruence subgroups 
\begin{align*}
\Gamma(N)&:=\left\{ \left(\begin{smallmatrix} a&b\\ c&d\end{smallmatrix}\right)\in \SL_2(\Z): a\equiv d\equiv 1\pmod{N},\ c\equiv d\equiv 0\pmod{N}\right\}\\
\Gamma_1(N)&:=\left\{ \left(\begin{smallmatrix} a&b\\ c&d\end{smallmatrix}\right)\in \SL_2(\Z): a\equiv d\equiv 1\pmod{N},\ c\equiv 0\pmod{N}\right\}\\
\Gamma_0(N)&:=\left\{ \left(\begin{smallmatrix} a&b\\ c&d\end{smallmatrix}\right)\in \SL_2(\Z): c\equiv 0\pmod{N}\right\}.
\end{align*}
The index of each of these subgroups in $\SL_2(\Z)$ is given in the following lemma.
\begin{lemma}\label{lem:GammaIndex}
\noindent

\noindent
\begin{enumerate}[leftmargin=*,label={\rm(\arabic*)}]
\item For $N\in\N$ we have
\begin{equation}\label{eqn:Gamma0index}
\left[\SL_2(\Z):\Gamma_0(N)\right]=N\prod_{p\mid N} \left(1+\frac{1}{p}\right).
\end{equation}
\item For $N\in\N$ we have
\begin{equation}\label{eqn:Gammaindex}
\left[\SL_2(\Z):\Gamma(N)\right]=N^3\prod_{p\mid N} \left(1-\frac{1}{p^2}\right).
\end{equation}
\item If $L\mid N$, then 
\begin{equation}\label{eqn:Gamma01index}
\left[\SL_2(\Z):\Gamma_0(N)\cap\Gamma_1(L)\right]=N\varphi(L)\prod_{p\mid N} \left(1+\frac{1}{p}\right),
\end{equation}
where $\varphi$ is the Euler totient function.
\end{enumerate}
\end{lemma}

Elements $\gamma=\left(\begin{smallmatrix}a&b\\c&d\end{smallmatrix}\right)$ of $\SL_2(\Z)$ act on the complex upper half-plane $\H:=\{\tau\in\C:\im(\tau)>0\}$ via \begin{it}fractional linear transformations\end{it}
\[
f(\gamma\tau):=f\left(\frac{a\tau+b}{c\tau+d}\right).
\]
For a congruence subgroup $\Gamma$, the elements of $\Gamma\backslash (\Q\cup\{i\infty\})$ are called the \begin{it}cusps of $\Gamma$\end{it}. We generally associate the cusps with a full set of representatives in this quotient and do not distinguish between the coset and the representative. The number of cusps for $\Gamma(N)$ is given in the following lemma.
\begin{lemma}\label{lem:cuspsGammaN}
Let $N\in\N$ be given. For every $\delta\mid N$, there are $\varphi\left(\frac{N}{\delta}\right)\varphi(\delta)\frac{N}{\delta}$ cusps $\varrho=\gamma(i\infty)$ modulo $\Gamma(N)$ with $\gamma=\left(\begin{smallmatrix}a&b\\ c&d\end{smallmatrix}\right)\in \SL_2(\Z)$ for which $(c,N)=\delta$. 
\end{lemma}
\begin{proof}
Suppose that $\delta\mid c$ and set $c':=\frac{c}{\delta}$. Writing 
\[
\varrho=\frac{a}{c}=\frac{a}{\delta c'},
\]
we have $(c,N)=\delta$ if and only if $\gcd(c',N/\delta)=1$. There are hence $\varphi(N/\delta)$ choices of $c'$. Writing $a=\delta m +r$, we have $\gcd(r,\delta)=1$ and there are $\frac{N}{\delta}$ choices of $m$ modulo $N$. This gives the claim. 
\end{proof}

\subsection{Modular forms}

Let $k$ be an integer and $\Gamma\subseteq\SL_2(\Z)$ be a congruence subgroup of $\SL_2(\Z)$. We also let $\chi$ be a character from $\Gamma$ to $\C$ with $|\chi(\gamma)|=1$ for all $\gamma\in\Gamma$. This is often given by a Dirichlet character $\chi(\gamma):=\chi(d)$, where $\gamma=\left(\begin{smallmatrix}a&b\\c&d\end{smallmatrix}\right)$, and we abuse notation to write $\chi$ in both settings. For $k\in\Z$, we say that a function $f$ satisfies  \begin{it}weight $k$ modularity on $\Gamma$ with (Nebentypus) character $\chi$\end{it} if 
\[
f(\gamma\tau)= \chi(\gamma) (c\tau+d)^k f(\tau)
\]
for every $\gamma=\left(\begin{smallmatrix}a&b\\c&d\end{smallmatrix}\right)\in\Gamma$. Defining the \begin{it}weight $k$ slash operator\end{it} $|_k$ by 
\[
f\big|_k\gamma(\tau):=(c\tau+d)^{-k}f(\gamma\tau),
\]
the modular property can be written as
\[
f\big|_k\gamma=\chi(\gamma)f.
\]
A \begin{it}weight $k$ (holomorphic) modular form on $\Gamma$ with character $\chi$\end{it} is a holomorphic function $f:\H\to\C$ which satisfies weight $k$ modularity on $\Gamma$ with character $\chi$ and for every cusp $\frac{a}{c}\in\Q$ the function $f\left(\frac{a}{c}+iy\right)$ grows at most polynomially in $y^{-1}$ as $y\to 0$. We denote the space of modular forms of weight $k$ on $\Gamma$ with character $\chi$ by $M_k(\Gamma,\chi)$. If $f$ vanishes towards every cusp, then we call $f$ a \begin{it}cusp form\end{it} and we let $S_k(\Gamma,\chi)$ denote the space of cusp forms of weight $k$ on $\Gamma$ with character $\chi$. 

Since $\left(\begin{smallmatrix}1&N\\ 0&1\end{smallmatrix}\right)\in\Gamma$ for some $N\in\N$, we have $f(\tau+N)=f(\tau)$ for a modular form $f$ of weight $k$ on $\Gamma$, and hence we have a Fourier expansion
\[
f(\tau)=\sum_{n\in\N_0} c_f(n) q^{\frac{n}{N}}.
\]
There is a similar expansion 
\begin{equation}\label{eqn:Fourier}
f_{\varrho}(\tau):=f\big|_{k}\gamma_{\varrho}^{-1}(\tau) = \sum_{n\in\N_0} c_{f,\varrho}(n) q^{\frac{n}{N_{\varrho}}}
\end{equation}
at other cusps $\varrho\in \Q$, where $\gamma_{\varrho}(\varrho)=i\infty$ and $N_{\sigma}\in\N$ is the \begin{it}cusp width\end{it}. Our primary focus in this paper is the investigation of Fourier expansions coming from the generating functions for $r_{m,\bma}^*(n)$. We recall the connection in the next subsection.

\subsection{Petersson inner products and Eisenstein series}
We require the \begin{it}Petersson inner product\end{it} on $f,g\in S_k(\Gamma,\chi)$ via (defined for $\tau=u+iv$)
\begin{equation}\label{eqn:innerdef}
\left<f,g\right>:=\frac{1}{\left[\SL_2(\Z): \Gamma\right]} \int_{\Gamma\backslash\H} f(\tau) \overline{g(\tau)} v^{k} \frac{du dv}{v^2}.
\end{equation}
The \begin{it}Petersson norm\end{it} $\|f\|:=\sqrt{\left<f,f\right>}$ of $f$ plays a crucial role in bounding the coefficients of certain cusp forms in Section \ref{sec:cuspidal}. Moreover, using the Petersson inner product, one can define the space of \begin{it}Eisenstein series\end{it} to be the set of holomorphic modular forms which are orthogonal to all cusp forms under the Petersson inner product.

\subsection{Setup and relation with quadratic forms with congruence conditions}\label{sec:QuadFormRel}
By multiplying by $8(m-2)$ on both sides and completing the square, we see that every solution $\bm{x}\in\Z^{\ell}$ of \eqref{eqn:sumsmgonal} corresponds to a solution to the equation 
\[
8(m-2)n+\sum_{j=1}^{\ell}\aj_j (m-4)^2 = \sum_{j=1}^{\ell}\aj_jx_j^2
\]
with $\bm{x}\in\Z^{\ell}$ satisfying $x_j\equiv 4-m\equiv m\pmod{2(m-2)}$.

Thus, setting  
\[
s_{r,M,\bma}^*(n):=\#\left\{ \bm{x}\in\Z^{\ell}: \sum_{1\leq j\leq \ell} \aj_j x_j^2=n,\ x_j\equiv r\pmod{M}\right\},
\]
we have 
\begin{equation}\label{eqn:r*s*rel}
r_{m,\bma}^*(n)= s_{m,2(m-2),\bma}^*\left(8(m-2)n + \sum_{j=1}^{\ell} \aj_j(m-4)^2\right).
\end{equation}
We now take the generating function 
\begin{equation}\label{eqn:Thetadef}
\Theta_{r,M,\bma}^*(\tau):=\sum_{n\geq 0} s_{r,M,\bma}^*(n) q^{\frac{n}{M}}. 
\end{equation}
The modular properties of $\Theta_{r,M,\bma}^*$ are given in \cite[Proposition 2.1]{Shimura}. To state it, we let $V_M$ denote the usual $V$-operator $f| V_M(z):=f(Mz)$ and for a discriminant $\Delta$ and $d\in\Z$ we let $\chi_\Delta(d):=\left(\frac{\Delta}{d}\right)$ denote the Kronecker--Legendre--Jacobi symbol. 
\begin{lemma}\label{lem:ThetaModular}
For $\bma\in \N^4$, we have
\[
\Theta_{r,M,\bma}^*\big| V_M \in M_{2}\left(\Gamma_0\left(4\operatorname{lcm}(\bma)M^2\right)\cap \Gamma_1(M), \chi_{4\prod_{j=1}^4 \aj_j}\right).
\] 
\end{lemma}
\begin{extradetailsproof}
\begin{proof}
For a quadratic form $Q$ of rank $d$ we define the \begin{it}Hessian matrix\end{it} $A_Q$ to be the unique $d\times d$ matrix such that 
\[
Q(\bm{x})=\frac{1}{2}\bm{x}^TA_Q\bm{x}.
\]
Recall that the \begin{it}level $N_Q$ of $Q$\end{it} is defined to be the smallest $N\in\N$ such that $NA_Q^{-1}$ has integer coefficients with even diagonal entries. We moreover define the \begin{it}discriminant\end{it} $D_Q$ of $Q$ to be the determinant of $A_Q$.  By \cite[Theorem 2.4]{Cho}, if $Q$ is a rank $d$ ($d$ even) quadratic form, $\bm{h}\in\Z^d$ and $M\in\N$, then 
\begin{equation}\label{eqn:ChoThetaModular}
\sum_{\substack{\bm{x}\in \Z^d\\ \bm{x}\equiv \bm{h}\pmod{M}}}q^{Q(\bm{x})}\in M_{\frac{d}{2}}\left(\Gamma_0\left(N_Q M^2\right)\cap\Gamma_1(M),\left(\frac{D_Q}{\cdot}\right)\right).
\end{equation}
We take $h_j=r$ for all $j$ and note that 
\[
\Theta_{r,M,\bma}^*(M\tau) =\sum_{n\geq 0} s_{r,M,\bma}^*(n) q^n= \sum_{n\geq 0} \sum_{\substack{\bm{x}\in \Z^d\\ x_j\equiv r\pmod{M}\\ \sum_{j=1}^4\aj_j x_j^2=n}} q^n= \sum_{\substack{\bm{x}\in \Z^d\\ x_j\equiv r\pmod{M}}} q^{\sum_{j=1}^4\aj_j x_j^2}.
\]
This is precisely the theta function from \eqref{eqn:ChoThetaModular} for the quadratic form $Q(\bm{x}):= \sum_{j=1}^4\aj_j x_j^2$ and $\bm{h}=(r,r,r,r)$. Note that 
\[
A_Q=\left(\begin{smallmatrix} 2\aj_1&0 &0 &0\\ 0& 2\aj_2&0 &0\\ 0&0& 2\aj_3&0 \\ 0&0&0& 2\aj_4\end{smallmatrix}\right).
\]
Therefore $D_Q=16\prod_{j=1}^4\aj_j$, 
\[
A_Q^{-1}=\left(\begin{smallmatrix} \frac{1}{2\aj_1}&0 &0 &0\\ 0& \frac{1}{2\aj_2}&0 &0\\ 0&0& \frac{1}{2\aj_3}&0 \\ 0&0&0& \frac{1}{2\aj_4}\end{smallmatrix}\right),
\]
and we see that $N_Q=4\operatorname{lcm}(\bma)$. We conclude from \eqref{eqn:ChoThetaModular} that
\[
\Theta_{r,M,\bma}^*\big|V_M\in M_{2}\left(\Gamma_0\left(4\operatorname{lcm}(\bma)M^2\right)\cap \Gamma_1(M), \left(\frac{\prod_{j=1}^4 \aj_j}{\cdot}\right)\right).
\]
\noindent

\end{proof}
\end{extradetailsproof}

There is a natural decomposition of modular forms 
\begin{equation}\label{eqn:decomposition}
\Theta_{r,M,\bma}^*\big|V_M=E_{r,M,\bma}+f_{r,M,\bma}.
\end{equation}
Here $E_{r,M,\bma}$ is an Eisenstein series (the so-called ``main term'' from the introduction) and $f_{r,M,\bma}$ is a cusp form (the ``error term'' from the introduction). Since 
\begin{equation}\label{eqn:EisCuspSplit}
s_{r,M,\bma}^*(n)=A_{r,M,\bma}(n)+B_{r,M,\bma}(n), 
\end{equation}
the goal in this paper is to show that $A_{r,M,\bma}(n)>|B_{r,M,\bma}(n)|$ for $n$ sufficiently large (with an explicit bound on $n$), and hence $s_{r,M,\bma}^*(n)>0$.

\subsection{Some explicit bounds}
We require some useful elementary bounds. Expanding the geometric series and differentiating, one obtains the following result by induction.
\begin{lemma}\label{lem:expsum}
For any $k\in\N_{0}$ we have 
\[
\sum_{n=1}^{\infty} n^ke^{-\frac{2\pi n}{N}}\leq \frac{k!}{\left(1-e^{-\frac{2\pi}{N}}\right)^{k+1}}.
\]
Moreover, defining
\[
\delta_N:=\begin{cases}1 &\text{if }N\geq 4,\\ \frac{1}{2}&\text{if }N\in \{2,3\}\\ \frac{1}{4}&\text{if }N=1.
\end{cases}
\]
we have 
\[
\sum_{n=1}^{\infty} n^ke^{-\frac{2\pi n}{N}}\leq \frac{k!}{(\delta_N\pi) ^{k+1}} N^{k+1}.
\]
\end{lemma}
The following result, which follows by a trick of Ramanujan, can be found in \cite[Lemma 2.3]{BanerjeeKane}.
\begin{lemma}\label{lem:sigma0N}
For $N\in\N$ we have 
\[
\sigma_0(N)\leq \mathcal{C}_{\alpha} N^{\alpha},
\]
where 
\[
\mathcal{C}_{\alpha}:=\prod_{p<2^{\frac{1}{\alpha}}}\max\left\{\frac{j+1}{p^{j\alpha}}:j\geq 1\right\}.
\]
Specifically, we have 
\[
\mathcal{C}_{\frac{1}{10}}<4.175\times 10^{10},\qquad\qquad \mathcal{C}_{\frac{1}{15}}<2.751\times 10^{120}.
\]
\end{lemma}

\subsection{Quadratic forms}
A \begin{it}quadratic form\end{it} over $\Q$ of rank $\ell$ is a (non-degenerate) homogenous polynomial of degree two in $\ell$ variables and can be written (for some $\aj_{ij}\in\Q$)
\begin{equation}\label{eqn:Qdef}
Q(\bm{x})=\sum_{1\leq i\leq j\leq \ell} \aj_{ij} x_ix_j.
\end{equation}
We usually assume that $\aj_{ij}\in\Z$ and moreover assume that $\aj_{ij}\in 2\Z$ if $i\neq j$, which we may assume because our forms in Section \ref{sec:QuadFormRel} and we only consider quadratic forms which are equvalent over the $2$-adic integers $\Z_2$; such quadratic forms are called \begin{it}integral\end{it}. We furthermore assume that $Q$ is \begin{it}positive-definite\end{it}, which means that $Q(\bm{x})\geq 0$ and $Q(\bm{x})=0$ if and only if $\bm{x}=\bm{0}$. The corresponding \begin{it}Gram matrix\end{it} of $Q$ is given by the $\ell\times \ell$ symmetric matrix $A_Q$ for which 
\[
Q(\bm{x})=\frac{1}{2}\bm{x}^T A_Q \bm{x}.
\]
Note that some authors write $Q(\bm{x})=\bm{x}^T A_Q\bm{x}$, dividing the Gram matrix by a factor of two. There is an induced symmetric bilinear form given by $B_Q(\bm{x},\bm{y}):=\bm{x}^T A_Q\bm{y}$. We omit the dependence on $Q$ when it is clear from the context. The \begin{it}discriminant\end{it} of $Q$ is given by 
\[
\Delta_Q:=\det(A_Q)
\]
and the \begin{it}level\end{it} $N_Q$ is chosen to be the minimal $N\in\N$ for which 
\[
N A_Q^{-1}
\]
has integer coefficients with even diagonal.

 Written in the theory of quadratic lattices, we have an $\ell$-dimensional quadratic space $V$ with the underlying quadratic form $\mathcal{Q}$. An \begin{it}isometry\end{it} is a bijective linear map $\sigma$ on $V$ which preserves the bilinear form induced by $\mathcal{Q}$, i.e., for any $\bm{u},\bm{v}\in V$ we have 
\[
B_{\mathcal{Q}}(\sigma(\bm{u}),\sigma(\bm{v}))=B_{\mathcal{Q}}(\bm{u},\bm{v}).
\]
For a lattice $L$ with generators $\bm{u_1},\dots,\bm{u_\ell}$, the quadratic form $\mathcal{Q}$ evaluated on this lattice will yield a quadratic form
\[
Q(\bm{x})=Q_L(\bm{x}):=\frac{1}{2}\sum_{j=1}^{\ell}\sum_{k=1}^{\ell} B_{\mathcal{Q}}\left(\bm{u_j},\bm{u_k}\right) x_jx_k
\]
Although $Q_L$ depends on the choice of generators, a change of basis map yields an isometry that is bijective on $L$ and the resulting quadratic forms are hence equivalent up to isometry. One can restrict the variables $\bm{x}$ in $Q$ to congruence conditions $\bm{x}\equiv \bm{r}\pmod{M}$ by taking a shifted lattice $L+\frac{\bm{v}}{M}$ for some vector $\bm{v}\in L$. Instead, taking $\bm{v}\in ML$ would yield the unrestricted quadratic form. Specifically, choosing $L=(M\Z)^{\ell}$ and $\bm{v}=M\bm{r}$ gives the quadratic form with congruence conditions
\[
\frac{1}{2}\sum_{j=1}^{\ell}\sum_{k=1}^{\ell} B_{\mathcal{Q}}\left(M x_j\bm{e_j}+\bm{r},Mx_k\bm{e_k}+\bm{r}\right),
\]
where $\bm{e_j}$ denote the standard basis of $\Z^{\ell}$.  The orbits of the shifted lattices $L+\frac{\bm{v}}{M}$ under the isometries defined over $\Z$ form equivalence classes.  We furthermore let $L^*$ denote the \begin{it}dual lattice of $L$\end{it}, which is the set of $\bm{u}\in V$ satisfying $2B_Q(\bm{u},\bm{v})\in \Z$ for every $\bm{v}\in L$, and write $\widetilde{L}:=2L^*$.

Taking $V\otimes\Q_p$, one can similarly define orbits under local isometries and the integral quadratic forms which are locally equivalent to a given $L+\frac{\bm{v}}{M}$ for all $p$ are called the \begin{it}genus\end{it} of $L+\frac{\bm{v}}{M}$. Let $\GG_{\bm{r},M}(Q)$ furthermore denote a set of representatives of the classes in the genus of $L+\frac{\bm{r}}{M}$ and $w_{L'+\frac{\bm{v}'}{M}}$ denote the number of automorphs of $L'+\frac{\bm{v}'}{M}$ and define the theta function
\[
\Theta_{L'+\frac{\bm{v}'}{M}}(\tau):=\sum_{\bm{x}\in L'+\frac{\bm{v'}}{M}} q^{Q_{\bma}(\bm{x})}.
\]
The Siegel--Weil average is then given by
\begin{equation}\label{eqn:SiegelWeil}
E_{Q,\bm{r},M}=\frac{1}{\sum_{L'+\frac{\bm{v}'}{M}\in\GG_{\bm{r},M}(Q)}w_{L'+\frac{\bm{v}'}{M}}^{-1}}  \sum_{L'+\frac{\bm{v}'}{M}\in\GG_{\bm{r},M}(Q)} \frac{\Theta_{L'+\frac{\bm{v'}}{M}}}{w_{L'+\frac{\bm{v}'}{M}}}. 
\end{equation}
For $M=1$, the first identity is due to Siegel \cite{SiegelQuadForm} and a generalization by Weil \cite{Weil}, while the generalized form may be found in Xu Fei's work \cite{Xu}.
The Siegel--Weil average is the Eisenstein series component of the theta function. In particular, in the case that we investigate \eqref{eqn:SiegelWeil} is precisely the Eisenstein series part $E_{r,M,\bma}$ of the decomposition \eqref{eqn:decomposition} of the theta function.

Siegel \cite{SiegelQuadForm} and Weil \cite{Weil} (in the case of lattices) and van der Blij \cite{vanderBlij} and Shimura \cite{ShimuraCongruence} (in the case of shifted lattices) computed the coefficients of the Eisenstein series as a product of \begin{it}local densities\end{it} from genus theory. Namely, setting 
\[
\beta_{Q,p,r,M}(n):=\lim_{U\to \{n\}} \frac{\vol_{\Z_p^{\ell}}(Q^{\leftarrow}(U))}{\vol_{\Z_p}(U)},
\]
where $U$ is an open $\Z_p$-ball around $n$, we have
\begin{equation}\label{eqn:localdensity}
A_{r,M,\bma}(n)=\beta_{Q,\infty}(n)\prod_{p} \beta_{Q,p,r,M}(n).
\end{equation}
where $\beta_{Q,\infty}(n)$ is defined analogously by taking balls in $\R$. We omit the dependence on $r$ and $M$ whenever $M=1$.

We require the following quantitative version of the uniform bound from \cite[Lemma 4.1]{Blomer}. 
\begin{lemma}\label{lem:Blomer}
Let $Q$ be a quadratic form, as in \eqref{eqn:Qdef} with $a_{ij}\in\Q$, $M\in\N$, and $\bm{r}\in \Z^{\ell}$ be given. Then for 
\[
R_{Q,\bm{r},M}(n):=\#\{\bm{x}\in\Z^{\ell}: Q(\bm{x})\leq n,\, \bm{x}\equiv \bm{r}\pmod{M}\},
\]
we have 
\[
R_{Q,\bm{r},M}(n)\leq \frac{\left(\frac{2}{M}+1\right)^{\ell}(2n)^{\frac{\ell}{2}}}{\sqrt{\Delta_Q}} + \ell\left(\frac{2}{M}+1\right)^{\ell-1} n^{\frac{\ell-1}{2}}.
\]
\end{lemma}
\begin{proof}
Every quadratic form
\[
Q(\bm{x})=\sum_{1\leq i\leq j\leq \ell} a_{ij} x_ix_j
\]
 is isometric over $\Z$ to a Minkowski-reduced form 
\[
Q'(\bm{x})=\sum_{1\leq i\leq j\leq \ell}a_{ij}'x_ix_j
\]
 with $\aj_{11}'\leq \aj_{22}'\leq \dots\leq \aj_{\ell\ell}'$. The congruence conditions may change under the isometry, but are still of the form $\bm{x}\equiv \bm{r}'\pmod{M}$ for some $\bm{r}'$. We assume without loss of generality that $|r_j'|\leq \frac{M}{2}$. Since we use no additional properties of $\bm{r}'$ below, we may assume that our original choices of $Q$ and $\bm{r}$ give the Minkowski-reduced form and simply write $Q$ and $\bm{r}$ for $Q'$ and $\bm{r}'$ below. 

We claim that $|\aj_{ij}|\leq \aj_{ii}$ for all $i$ and $j$. By the definition of a Minkowski-reduced form, for any $\bm{u}\in\Z^{\ell}$ with $\gcd(u_i,u_{i+1},\dots,u_{\ell})=1$, we have 
\begin{equation}\label{eqn:MinkowskiReduced}
Q(\bm{u})\geq Q(\bm{e}_i),
\end{equation}
where $\bm{e}_i$ is the canonical basis element with $\bm{e}_{i,j}=\delta_{i=j}$. We take $\bm{u}=\bm{e}_i\pm \bm{e}_j$, yielding 
\[
\aj_{jj}=Q(\bm{e}_j) \leq Q\left(\bm{e}_i\pm\bm{e}_j\right) = \aj_{ii}\pm \aj_{ij}+\aj_{jj},
\]
from which we conclude that $\mp \aj_{ij}\leq \aj_{ii}$, or in other words $|\aj_{ij}|\geq \aj_{ii}$ for all $i$ and $j$.

We prove the claim by induction on $\ell$. For $\ell=1$ we have $Q(x)=ax^2$ for some $a$ and $\Delta_Q=2a$. Writing $x=My+r$, we see that $Q(x)\leq n$ if and only if $|My+r|=|x|\leq \sqrt{\frac{n}{a}}=\sqrt{\frac{2n}{\Delta_Q}}$. This becomes 
\[
-\frac{1}{M} \sqrt{\frac{2n}{\Delta_Q}}-\frac{r}{M}\leq y\leq \frac{1}{M}\sqrt{\frac{2n}{\Delta_Q}}-\frac{r}{M}.
\]
The number of integers in this interval is at most 
\[
1+\frac{2}{M}\sqrt{\frac{2n}{\Delta_Q}}.
\]

Now suppose that the claim is true for $\ell-1\geq 1$. Consider 
\begin{align*}
Q(\bm{x})&=\sum_{1\leq i\leq j\leq \ell} \aj_{ij} x_ix_j=\aj_{11}\left(x_1+\frac{\sum_{j=2}^{\ell} \aj_{1j}x_j}{2\aj_{11}}\right)^2 + \sum_{2\leq i\leq j\leq \ell} \aj_{ij}x_ix_j-\frac{1}{4\aj_{11}}\left(\sum_{j=2}^{\ell} \aj_{1j}x_j\right)^2\\
&=\aj_{11}\left(x_1+\frac{\sum_{j=2}^{\ell} \aj_{1j}x_j}{2\aj_{11}}\right)^2 + \sum_{2\leq i\leq j\leq \ell} \left(\aj_{ij}-\frac{\delta_{i\neq j}}{2\aj_{11}} \aj_{1i}\aj_{1j}-\frac{\delta_{i=j}}{4\aj_{11}}\aj_{1j}^2\right)x_ix_j\\
&=\aj_{11}\left(x_1+\frac{\sum_{j=2}^{\ell} \aj_{1j}x_j}{2\aj_{11}}\right)^2 + \widetilde{Q}\!\left(x_2,\dots,x_{\ell}\right).
\end{align*}
Set $\widetilde{\bm{x}}:=(x_2,\dots,x_{\ell})$ and $x_1':=x_1'(\widetilde{\bm{x}}):=x_1+\frac{\sum_{j=2}^{\ell} \aj_{1j}x_j}{2\aj_{11}}$. We have $x_1'^2\leq \frac{n}{\aj_{11}}$. Applying the congruence condition $x_1\equiv r_1\pmod{M}$ to rewrite $x_1=My_1+r_1$, we have 
\begin{multline*}
-\sqrt{\frac{n}{\aj_{11}}}\leq  My_1 + r_1 + \frac{\sum_{j=2}^{\ell} \aj_{1j}x_j}{2\aj_{11}}\leq \sqrt{\frac{n}{\aj_{11}}}\\
\Leftrightarrow  -\frac{1}{M} \sqrt{\frac{n}{\aj_{11}}}-\frac{r_1}{M} - \frac{\sum_{j=2}^{\ell} \aj_{1j}x_j}{2M \aj_{11}} \leq  y_1 \leq \frac{1}{M}\sqrt{\frac{n}{\aj_{11}}}  -\frac{r_1}{M} - \frac{\sum_{j=2}^{\ell} \aj_{1j}x_j}{2M\aj_{11}}
\end{multline*}
For each fixed $\widetilde{\bm{x}}$, there are hence at most $\frac{2}{M}\sqrt{\frac{n}{\aj_{11}}}+1$ choices of $y_1$ such that $|x_1'|\leq \sqrt{\frac{n}{\aj_{11}}}$. We claim that 
\begin{equation}\label{eqn:Delrel}
\Delta_{\widetilde{Q}}=\frac{\Delta_{Q}}{2\aj_{11}}.
\end{equation}
To see \eqref{eqn:Delrel}, note that $\Delta_Q$ is the determinant of the matrix
\[
A=A_Q=\left(\frac{2 \aj_{ij}}{1+\delta_{i\neq j}}\right)
\]
Applying the row operations $R_1':=R_1$ and $R_j':=R_j-\frac{\aj_{1j}}{2\aj_{11}}R_{1}$ for $j\geq 2$ to the rows $R_1,\dots, R_{\ell}$, the new matrix $A'$ has the same determinant as $A$. The rows $R_2'$ to $R_{\ell}'$ all have $0$ in the first column and the other $\ell-1$ columns are precisely the Gram matrix $A_{\widetilde{Q}}$, yielding \eqref{eqn:Delrel}. 

Using \eqref{eqn:Delrel}, by induction there are at most 
\begin{multline*}
\frac{\left(\frac{2}{M}+1\right)^{\ell-1}}{\sqrt{\Delta_{\widetilde{Q}}}}(2n)^{\frac{\ell-1}{2}}+(\ell-1)\left(\left(\frac{2}{M}+1\right)\sqrt{n}\right)^{\ell-2}\\
 = \frac{\left(\frac{2}{M}+1\right)^{\ell-1}\sqrt{2\aj_{11}}}{\sqrt{\Delta_{Q}}}(2n)^{\frac{\ell-1}{2}}+\left(\frac{2}{M}+1\right)^{\ell-2}(\ell-1)(\sqrt{n})^{\ell-2}
\end{multline*}
points $\widetilde{\bm{x}}$ with $Q(\widetilde{\bm{x}})\leq n$, and for each of these there are at most $\frac{2\sqrt{n}}{M\sqrt{\aj_{11}}}+1$ choices of $x_1'$. We thus obtain that there are at most 
\begin{multline*}
\left(\frac{2\sqrt{n}}{M\sqrt{\aj_{11}}}+1\right)\left( \frac{\left(\frac{2}{M}+1\right)^{\ell-1}\sqrt{2\aj_{11}}}{M^{\ell-1}\sqrt{\Delta_{Q}}}(2n)^{\frac{\ell-1}{2}}+\left(\frac{2}{M}+1\right)^{\ell-2}(\ell-1)(\sqrt{n})^{\ell-2}\right)=\frac{\left(\frac{2}{M}+1\right)^{\ell}(2n)^{\frac{\ell}{2}}}{\sqrt{\Delta_Q}}\\
 + \left(\frac{2}{M}+1\right)^{\ell-1}\!\!\ell (\sqrt{n})^{\ell-1}\!\left(\frac{\ell-1}{\left(\frac{2}{M}+1\right)\ell\sqrt{n}}\left(1+\frac{2\sqrt{n}}{M\sqrt{\aj_{11}}}\right) + \frac{1}{\ell \sqrt{\Delta_Q}} \left(\sqrt{\aj_{11}} - \sqrt{n}\right)\right).
\end{multline*}
For $n\geq \aj_{11}$, we have $\sqrt{\aj_{11}} - \sqrt{n}\leq 0$ and hence the terms inside the parentheses in the second term may be bounded by (noting that $1\leq \frac{\sqrt{n}}{\sqrt{\aj_{11}}}$)
\begin{multline*}
\frac{\ell-1}{\left(\frac{2}{M}+1\right)\ell\sqrt{n}}\left(1+\frac{2\sqrt{n}}{M\sqrt{\aj_{11}}}\right) + \frac{1}{\ell \sqrt{\Delta_Q}} \left(\sqrt{\aj_{11}}- \sqrt{n}\right)\\
\leq \frac{\ell-1}{\left(\frac{2}{M}+1\right)\ell\sqrt{n}}\left( 1+\frac{2\sqrt{n}}{M\sqrt{\aj_{11}}}\right)\leq \frac{\ell-1}{\left(\frac{2}{M}+1\right)\ell\sqrt{n}}\left(\frac{\left(\frac{2}{M}+1\right)\sqrt{n}}{M\sqrt{\aj_{11}}}\right)=\frac{\ell-1}{\ell}\cdot \frac{1}{\sqrt{\aj_{11}}} <  1,
\end{multline*}
 so we are done as long as $n\geq \aj_{11}$. However, it is well-known that for $\bm{x}\neq 0$ we have $Q(\bm{x})\geq Q(\bm{e}_1)=\aj_{11}$, so for $n<\aj_{11}$ we have $R_{Q,\bm{r},M}(n)\leq R_Q(n)=1$.
\begin{extradetailsproof}
We claim that in the case $n<\aj_{11}$ we have $R_Q(n)=1$ (i.e., only the zero vector $\bm{0}$ has $Q(\bm{x})<\aj_{11}$). Suppose for contradiction that $Q(\bm{x})<\aj_{11}$. If $\gcd(\bm{x})\neq 1$, then setting $\widetilde{\bm{x}}:=\frac{1}{\gcd(\bm{x})}\bm{x}$, we see that $Q(\widetilde{\bm{x}})<Q(\bm{x})<\aj_{11}$, so without loss of generality we may assume $\gcd(\bm{x})=1$. Thus, since we have a Minkoswki-reduced basis, taking $i=1$ in \eqref{eqn:MinkowskiReduced}, we have 
\[
Q(\bm{x})\geq Q(e_1)=\aj_{11},
\]
a contradiction. Thus there are no $\bm{x}\in \Z^{\ell}\setminus\{0\}$ with $Q(\bm{x})\leq n$. 
\end{extradetailsproof}
\end{proof}

\section{Cuspidal part}\label{sec:cuspidal}
In this section, we bound the coefficients of the cusp forms occurring in the decomposition \eqref{eqn:decomposition}. We begin by recalling an explicit bound on the Fourier coefficients of an arbitrary cusp form in terms of its Petersson norm, the weight and the subgroup (see \cite[Lemma 4.1]{BanerjeeKane}). 
\begin{lemma}\label{lem:cf(n)<norm}
 Suppose that $f\in S_{k}(\Gamma_0(N)\cap\Gamma_1(L),\psi)$ with $L\mid N$ and $\psi$ a character modulo $N$. If $f$ has the Fourier expansion $f(\tau)=\sum_{n\geq 1} c_f(n) q^n,$ then for any $\delta>0$ we have 
\begin{equation}\label{eqn:afbnd}
\left|c_{f}(n)\right|\leq \sqrt{\frac{\pi k}{3}} e^{2\pi} \zeta(1+4\delta)^{\frac{1}{2}}c_{\delta}^{\frac{5}{2}}\sigma_0(n)n^{\frac{k-1}{2}}\|f\| N^{1+2\delta}\prod_{p\mid N}\left(1+\frac{1}{p}\right)^{\frac{1}{2}}\varphi(L).
\end{equation}
In particular, for $k=2$, we have 
\begin{align}\label{eqn:k=2spec1}
\left|c_{f}(n)\right|&\leq 4.58\cdot 10^{128}\cdot  n^{\frac{17}{30}}\|f\| N^{1+2\cdot 10^{-6}}\prod_{p\mid N}\left(1+\frac{1}{p}\right)^{\frac{1}{2}} \varphi(L),\\
\label{eqn:k=2spec3}\left|c_{f}(n)\right|&\leq 4.39\cdot 10^{79}\cdot  n^{\frac{4}{7}}\|f\| N^{1+2\cdot 10^{-6}}\prod_{p\mid N}\left(1+\frac{1}{p}\right)^{\frac{1}{2}} \varphi(L),\\
\label{eqn:k=2spec2}\left|c_{f}(n)\right|&\leq 6.95\cdot 10^{18}\cdot  n^{\frac{3}{5}}\|f\| N^{1+2.5\times 10^{-6}}\prod_{p\mid N}\left(1+\frac{1}{p}\right)^{\frac{1}{2}}\varphi(L).
\end{align}
\end{lemma}
In light of Lemma \ref{lem:cf(n)<norm}, it remains to bound $\|f\|$ in order to obtain a bound on $c_f(n)$. To do so for $f=f_{r,M,\bma}$ from the decomposition \eqref{eqn:decomposition}, we use arithmetic information coming from the theory of quadratic forms to bound  $\|f_{r,M,\bma}\|$ in terms of the parameters $r$, $M$, and $\bma$. In the case $M=1$, this was done in \cite{BanerjeeKane} with explicit constants, using a method of Blomer \cite{Blomer}, who gives an asymptotic for $\|f\|$ in terms of $\prod_{j=1}^{\ell}\aj_j$ but without explicit constants. The case $M>1$ was considered in \cite{KL}, generalizing an argument of Duke \cite{DukeTernary}, but again without explicit constants, which we need for Theorem \ref{thm:gamma7bound}. Here we emulate the argument in \cite{BanerjeeKane} but generalize it for $M>1$.

In what follows, we define
\[
Q_{\bma}(\bm{x}):=\sum_{j=1}^{\ell} \aj_j  x_j^2.
\]
We abbreviate
\begin{equation}\label{eqn:Deltaadef}
\Delta_{\bma}:=\Delta_{Q_{\bma}}=2^{\ell}\prod_{j=1}^{\ell} \aj_j
\end{equation}
for the discriminant of $Q_{\bma}$. We consider $Q_{\bma}(\bm{x})$ with $x_j\equiv r\pmod{M}$. We furthermore use $N_{\bma}$ to abbreviate the level $N_{Q_{\bma}}$ of $Q_{\bma}$.
\begin{lemma}\label{lem:cuspbound}
Suppose that $\ell\geq 4$ is an even integer. Then we have 
\noindent

\noindent
\begin{enumerate}[leftmargin=*,label={\rm(\arabic*)}]
\item We have 
\begin{multline}\label{eqn:afQbndgen}
\left|B_{r,M,\bma}(n)\right|\leq \frac{3^{\ell-1}\left(\frac{\ell}{2}-2\right)!^{\frac{1}{2}}}{2^{\frac{\ell}{4}-\frac{3}{2}}\pi^{\frac{\ell}{2}}\delta_{M^2N_{\bma}}^{\frac{\ell-1}{2}}} \frac{M^{\ell-2}N_{\bma}^{\frac{\ell}{2}+2\delta}\sqrt{\frac{\pi \ell}{6}} e^{2\pi}\zeta(1+4\delta)^{\frac{1}{2}} c_{\delta}^{\frac{5}{2}}\varphi(M)}{\prod_{p\mid M, p\nmid N_{\bma}} \left(1-p^{-2}\right)^{\frac{1}{2}}\prod_{p\mid N_{\bma}} \left(1-p^{-1}\right)^{\frac{1}{2}}} \sigma_0(n) n^{\frac{\ell}{4}-\frac{1}{2}} \\
\times \Bigg( \sum_{\delta\mid M^2N_{\bma}}\varphi\left(\frac{M^2N_{\bma}}{\delta}\right)\varphi(\delta)\frac{M^2N_{\bma}}{\delta}\left(\frac{\gcd(M^2,\delta)}{M^2}\right)^{\ell}\\
 \sum_{m=0}^{\frac{\ell}{2}-2} \frac{(2\pi)^{-m}}{\left(\frac{\ell}{2}-2-m\right)!}(\ell-m-2)! \left(\frac{9}{ \Delta_{\bma}}(\ell-m-1)\frac{M^2N_{\bma}}{\pi\delta_{M^2N_{\bma}}} +\ell^2 \right)\Bigg)^{\frac{1}{2}}.
\end{multline}
\item 
In the special case that $\ell=4$, we have 
\begin{multline}\label{eqn:afQbndgenell4}
\left|B_{r,M,\bma}(n)\right|\leq \frac{54}{\pi^{2}\delta_{M^2N_{\bma}}^{\frac{3}{2}}} \frac{M^{2}N_{\bma}^{2+2\delta}\sqrt{\frac{2 \pi}{3}} e^{2\pi}\zeta(1+4\delta)^{\frac{1}{2}} c_{\delta}^{\frac{5}{2}}\varphi(M)}{\prod_{p\mid M, p\nmid N_{\bma}} \left(1-p^{-2}\right)^{\frac{1}{2}}\prod_{p\mid N_{\bma}} \left(1-p^{-1}\right)^{\frac{1}{2}}} \mathcal{C}_{\frac{1}{10}}  n^{\frac{3}{5}} \\
\times \left( \sum_{\delta\mid M^2N_{\bma}}\varphi\left(\frac{M^2N_{\bma}}{\delta}\right)\varphi(\delta)\frac{M^2N_{\bma}}{\delta}\left(\frac{\gcd(M^2,\delta)}{M^2}\right)^{4}\right)^{\frac{1}{2}} \left(\frac{27}{ \Delta_{\bma}}\frac{M^2N_{\bma}}{\pi\delta_{M^2N_{\bma}}} +16 \right)^{\frac{1}{2}}.
\end{multline}

\end{enumerate}
\item 
\end{lemma}

\begin{proof}
(1) By  Lemma \ref{lem:cf(n)<norm}, we need to bound $\|f_{r,M,\bma}\|$.  We plug \eqref{eqn:SiegelWeil} into the decomposition \eqref{eqn:decomposition} to obtain 
\[
f_{r,M,\bma}= \Theta_{r,M,\bma}^*\big|V_M - E_{r,M,\bma} = \frac{1}{\sum_{L'+\frac{\bm{v}'}{M}\in\GG_{r,M}(Q)}w_{L'+\frac{\bm{v}'}{M}}^{-1}}  \sum_{L'+\frac{\bm{v}'}{M}\in\GG_{r,M}(Q)} \frac{\Theta_{r,M,\bma}^*\big|V_M-\Theta_{L'+\frac{\bm{v'}}{M}}}{w_{L'+\frac{\bm{v}'}{M}}}. 
\]
Here we simply used the fact that $\Theta_{r,M,\bma}^*$ can be rewritten as a sum of the same type with each summand independent of the summation variable (and hence one is simply multiplying by $1$). 

By the triangle inequality, we have 
\begin{equation}\label{eqn:normSiegelWeil}
\left\|f_{r,M,\bma}\right\|\leq \frac{1}{\sum_{L'+\frac{\bm{v}'}{M}\in\GG_{r,M}(Q)}w_{L'+\frac{\bm{v}'}{M}}^{-1}}  \sum_{L'+\frac{\bm{v}'}{M}\in\GG_{r,M}(Q)} \frac{\left\|\Theta_{r,M,\bma}^*\big|V_M-\Theta_{L'+\frac{\bm{v'}}{M}}\right\|}{w_{L'+\frac{\bm{v}'}{M}}}. 
\end{equation}
The goal now is to bound 
\[
\left\|\Theta_{r,M,\bma}^*\big|V_M-\Theta_{L'+\frac{\bm{v'}}{M}}\right\|
\]
independent of $L'+\frac{\bm{v}'}{M}$, which by \eqref{eqn:normSiegelWeil} gives a bound on $\|f_{r,M,\bma}\|$.

To obtain a bound on $\left\|\Theta_{r,M,\bma}^*\big|V_M-\Theta_{L'+\frac{\bm{v'}}{M}}\right\|$, we use an argument of Blomer \cite{Blomer}. Specifically, we need an explicit version of \cite[Lemma 4.2]{Blomer}. For ease of notation, for $N=N_{\bma}$ we set
\begin{align*}
\mu_{0,1}(N,L)&:=\left[\SL_2(\Z):\Gamma_0(N)\cap\Gamma_1(N)\right],& \mu(N)&:=\left[\SL_2(\Z):\Gamma(N)\right],\\
\mathcal{F}_{0,1}(N,L)&:=\Gamma_0(N)\cap\Gamma_1(L)\backslash\H& \mathcal{F}(N)&:=\Gamma(N)\backslash\H,& \mathcal{F}:=\SL_2(\Z)\backslash\H.
\end{align*}
Suppressing the dependence on all variables, denote 
\begin{equation}\label{eqn:fcuspdef}
f:=\Theta_{r,M,\bma}^*-\Theta_{L'+\frac{\bm{v'}}{M}}.
\end{equation}
Letting $\{\gamma_1,\dots,\gamma_{\mu(N)}\}$ denote a set of representatives of $\SL_2(\Z)/\Gamma(N)$, we then have (note that we have normalized the inner product in \eqref{eqn:innerdef} so that the inner product is independent of the subgroup)
\begin{align}
\nonumber \|f\|^2 &= \frac{1}{\mu_{0,1}(N,L)}\int_{\mathcal{F}_{0,1}(N,L)}|f(\tau)|^2 v^{\frac{\ell}{2}} \frac{du dv}{v^2}= \frac{1}{\mu(N)}\int_{\mathcal{F}(N)}|f(\tau)|^2 v^{\frac{\ell}{2}} \frac{du dv}{v^2}\\
\nonumber &= \frac{1}{\mu(N)}\sum_{j=1}^{\mu(N)} \int_{\gamma_j\mathcal{F}}|f(\tau)|^2 v^{\frac{\ell}{2}} \frac{du dv}{v^2}= \frac{1}{\mu(N)}\sum_{j=1}^{\mu(N)} \int_{\mathcal{F}}|f(\gamma_j\tau)|^2 \im(\gamma_j\tau)^{\frac{\ell}{2}} \frac{du dv}{v^2}\\
\label{eqn:innerprod}&= \frac{1}{\mu(N)}\sum_{j=1}^{\mu(N)} \int_{\mathcal{F}}\left|f\big|_{\frac{\ell}{2}}\gamma_j(\tau)\right|^2 v^{\frac{\ell}{2}} \frac{du dv}{v^2}
\end{align}
The cusp width at each cusp is precisely $N$, so we may write the cusps as $\varrho_{1},\dots, \varrho_{\mu(N)/N}$. For the cusp $\varrho_j=\gamma_j(i\infty)$ and $T:=\left(\begin{smallmatrix}1&1\\ 0&1\end{smallmatrix}\right)$, the set
\[
\left\{\gamma_j T^r: 0\leq N-1\right\}
\]
gives a set of representatives of the cosets in $\SL_2(\Z)/\Gamma(N)$ which send $i\infty$ to $\varrho_j$. We may thus write \eqref{eqn:innerprod} as 
\begin{align*}
 \|f\|^2 &=\frac{1}{\mu(N)}\sum_{j=1}^{\mu(N)/N} \sum_{h=1}^N \int_{\mathcal{F}}\left|f\big|_{\frac{\ell}{2}}\gamma_jT^h(\tau)\right|^2 v^{\frac{\ell}{2}} \frac{du dv}{v^2}\\
&=\frac{1}{\mu(N)}\sum_{j=1}^{\mu(N)/N} \int_{\bigcup_{h=1}^N T^h\mathcal{F}}\left|f\big|_{\frac{\ell}{2}}\gamma_j(\tau)\right|^2 v^{\frac{\ell}{2}} \frac{du dv}{v^2}.
\end{align*}
Denoting the cusps with $\gcd(c,N)=\delta$ by $\varrho_{\delta,j}=\gamma_{\delta,j}(i\infty)$ (so $\gamma_{\delta,j}=\gamma_{\varrho_{\delta,j}}^{-1}$ as defined in \eqref{eqn:Fourier}), we then use Lemma \ref{lem:cuspsGammaN} to rewrite this as 
\[
\frac{1}{\mu(N)}\sum_{\delta\mid N} \sum_{j=1}^{\varphi\left(\frac{N}{\delta}\right)\varphi(\delta)\frac{N}{\delta}} \int_{\bigcup_{r=1}^N T^r\mathcal{F}}\left|f\big|_{\frac{\ell}{2}}\gamma_{\delta,j}(\tau)\right|^2 v^{\frac{\ell}{2}} \frac{du dv}{v^2}.
\]
Letting 
\begin{equation}\label{eqn:adeltajdef}
f_{\delta,j}(\tau):=f\big|_{\frac{\ell}{2}}\gamma_{\delta,j}(\tau)=\sum_{n=1}^{\infty} a_{\delta,j}(n) q^{\frac{n}{N}}
\end{equation}
be the Fourier expansion at the cusp $\varrho_{\delta,j}$ (as in \eqref{eqn:Fourier}) and noting that 
\[
\bigcup_{r=1}^N T^r\mathcal{F}\subseteq \left\{\tau: -\frac{1}{2}\leq u\leq N-\frac{1}{2}, v\geq \frac{1}{2}\right\},
\]
we then bound the integral by 
\begin{align}
\nonumber  \int_{\bigcup_{r=1}^N T^r\mathcal{F}}\left|f_{\delta_j}(\tau)\right|^2 v^{\frac{\ell}{2}} \frac{du dv}{v^2}&\leq \int_{\frac{1}{2}}^{\infty} \int_{-\frac{1}{2}}^{N-\frac{1}{2}}\left|\sum_{n=1}^{\infty} a_{\delta,j}(n) q^{\frac{n}{N}}\right|^2 v^{\frac{\ell}{2}} \frac{du dv}{v^2}\\
\nonumber &=\sum_{n=1}^{\infty}\sum_{m=1}^{\infty} a_{\delta,j}(n)\overline{a_{\delta,j}(m)} \int_{\frac{1}{2}}^{\infty} e^{-\frac{2\pi (n+m)v}{N}} v^{\frac{\ell}{2}-2} dv  \int_{-\frac{1}{2}}^{N-\frac{1}{2}}e^{2\pi i \frac{(n-m)}{N}u}  du\\
\nonumber &=N\sum_{n=1}^{\infty} |a_{\delta,j}(n)|^2 \int_{\frac{1}{2}}^{\infty} e^{-\frac{4\pi nv}{N}} v^{\frac{\ell}{2}-2} dv\\
\label{eqn:innerexpand}&=N\sum_{n=1}^{\infty} |a_{\delta,j}(n)|^2\left(\frac{N}{4\pi n}\right)^{\frac{\ell}{2}-1}\Gamma\left(\frac{\ell}{2}-1,\frac{2\pi n}{N}\right),
\end{align}
where $\Gamma(s,x):=\int_{x}^{\infty} t^{s-1}e^{-t} dt$ is the \begin{it}incomplete gamma function\end{it}.

We now require a bound on $|a_{\delta,j}(n)|$. In order to do so, we recall that 
\begin{align*}
f_{\delta,j}(\tau)&=f\big|_{\frac{\ell}{2}}\gamma_{\delta,j}(\tau)=\sum_{n=1}^{\infty} a_{\delta,j}(n) q^{\frac{n}{N}},\\
f&=\Theta_{r,M,\bma}^*\big|V_M-\Theta_{L'+\frac{\bm{v'}}{M}},
\end{align*}
so 
\[
f_{\delta,j}(\tau)=\Theta_{r,M,\bma}^*\big|V_M\big|_{\frac{\ell}{2}}\gamma_{\delta,j}- \Theta_{L'+\frac{\bm{v'}}{M}}\big|_{\frac{\ell}{2}}\gamma_{\delta,j}.
\]
Thus it suffices to bound the absolute value of the coefficients of the two theta functions at each cusp and then use the triangle inequality to bound against the sum. We use the calculations in \cite[Chapter IX]{Schoeneberg} to obtain a formula for the coefficients of these theta functions (also see the calculations in \cite{Shimura} leading up to \cite[Proposition 2.1]{Shimura}, which include the case when $\ell$ is odd; specifically, if a representative of the the cusp is $\frac{a}{c}$ with $c>0$ odd, then we \cite[the fourth formula on page 455]{Shimura}, while for $c$ even we use \cite[(2.7)]{Shimura}).

We first rewrite our theta function in terms of $\theta(\tau;\bm{r},A,M)$ defined in \cite[(2.0)]{Shimura} (we have trivial spherical polynomial $P$ and omit it in the notation throughout). This is given by (compare with \cite[Chapter IX, (16)]{Schoeneberg}, where the notation $\vartheta_{A,\bm{r}}(\tau)$ is used; we choose Shimura's notation because the dependence on $N$ is suppressed in \cite{Schoeneberg})
\[
\theta(\tau;\bm{r},A,M)(\tau)=\sum_{\bm{x}\equiv \bm{r}\pmod{M}}q^{\frac{1}{2M^2}\bm{x}^T A\bm{x}},
\]
where $A$ is a matrix for which both $A$ and $MA^{-1}$ have integral coefficients and $A\bm{r}\in M\Z^{\ell}$. We need to compare this with 
\[
\Theta_{Q,\bm{r},M}(\tau):=\sum_{\bm{x}\equiv \bm{r}\pmod{M}} q^{ Q(\bm{x})}
\]
for some quadratic form $Q$ in the genus of $Q_{\bma}$ and some vector $\bm{r}\in\Z^{\ell}$. Let $A_Q$ be the matrix for which $Q(\bm{x})=\frac{1}{2}\bm{x}^TA_Q\bm{x}$ and let $N_Q\in\N$ be chosen minimally so that $N_QA_Q^{-1}$ has integer coefficients. Then 
\begin{align*}
\Theta_{Q,\bm{r},M}(\tau)&=\sum_{\bm{x}\equiv \bm{r}\pmod{M}} q^{\bm{x}^TA_Q\bm{x}}\\
&=\sum_{\bm{x}\equiv \bm{r}\pmod{M}} q^{ \frac{1}{2N_Q^2M^4}N_Q^2M^2 \bm{x}^T\left(M^2A_Q\right)\bm{x}}\\
&=\sum_{\bm{x}\equiv MN_Q\bm{r}\pmod{N_QM^2}} q^{ \frac{1}{2N_Q^2M^4}\bm{x}^T\left(M^2A_Q\right)\bm{x}}\\
&=\theta\left(\tau;MN_Q\bm{r},M^2A_Q,M^2N_Q\right).
\end{align*}
Note that 
\[
(M^2A_Q)^{-1}=\frac{1}{M^2} A_Q^{-1},
\]
 so 
\[
(M^2N_Q)(M^2A_Q)^{-1} = N_QA_Q^{-1}
\]
has integer coefficients by construction. Since $\bm{r}\in\Z^{\ell}$ and $Q$ is an integral quadratic form (so $A_Q$ has integer coefficients), we have $A_Q \bm{r}\in\Z^{\ell}$, so the condition 
\[
M^2A_Q(MN_Q\bm{r})\in M^2N_Q\Z^{\ell}
\]
holds trivially. Thus 
\[
\Theta_{Q,\bm{r},M}(\tau)=\theta\left(\tau;MN_Q\bm{r},M^2A_Q,M^2N_Q\right)
\]
satisfies the conditions in \cite[(2.0)]{Shimura} and we may directly use the identities on \cite[page 455]{Shimura} (respectively, \cite[Chapter IX, (17)--(24)]{Schoeneberg}). We note that the determinant of $A_Q$ is the discriminant $\Delta_Q=\Delta_{\bma}$ and the level $N_Q$ is $N_{\bma}$. Since these are genus invariants, it's the same for every form in the genus. Thus the discriminant of $M^2A_Q$ is $M^{2\ell}\Delta_Q$ and the level of $M^2A_Q$ is $M^2N_Q$.

If $\varrho_j=\frac{a}{c}$ with $c>0$, then by \cite[Chapter IX, (24)]{Schoeneberg} we have 
\begin{align*}
\Theta_{Q,\bm{r},M}(\gamma_{\delta,j}\tau) &=\theta\left(\gamma_{\delta,j}\tau;MN_Q\bm{r},M^2A_Q,M^2N_Q\right)\\
& = M^{-\ell}\Delta_{Q}^{-\frac{1}{2}}c^{-\frac{\ell}{2}} (-i)^{\frac{\ell}{2}} (c\tau+d)^{\frac{\ell}{2}} \sum_{\substack{\bm{k}\pmod{M^2N_Q}\\ A_Q\bm{k}\equiv 0\pmod{N_Q}}} \Phi(M N_Q\bm{r},\bm{k})\theta\left(\tau;\bm{k},M^2A_Q,M^2N_Q\right)
\end{align*}
with 
\begin{equation}\label{eqn:Phidef}
\Phi\left(MN_Q\bm{r},\bm{k}\right):=\sum_{\substack{\bm{g}\pmod{cM^2N_Q}\\ \bm{g}\equiv MN_Q\bm{r}\pmod{M^2N_Q}}} e^{\frac{2\pi i}{cM^4N_Q^2}\left(\frac{a}{2}\bm{g}^TM^2A_Q \bm{g} + \bm{k}^T M^2A_Q\bm{g}+\frac{d}{2}\bm{k}^TM^2A_Q\bm{k}\right)}.
\end{equation}
Here $\bm{g}$ runs through elements of $\Z^{\ell}$ modulo $cM^2N_Q$ satisfying the given congruence. So 
\begin{equation}\label{eqn:CuspExpansionOdd}
\Theta_{Q,\bm{r},M}\big|_{\frac{\ell}{2}} \gamma_{\delta,j}(\tau) = M^{-\ell} \Delta_{Q}^{-\frac{1}{2}} c^{-\frac{\ell}{2}} (-i)^{\frac{\ell}{2}} \sum_{\substack{\bm{k}\pmod{M^2N_Q}\\ A_Q\bm{k}\equiv 0\pmod{N_Q}}} \Phi(MN_Q\bm{r},\bm{k})\theta\left(\tau;\bm{k},M^2A_Q,M^2N_Q\right). 
\end{equation}
By making a change of variables $\bm{g}=MN_Q\bm{r}+ \bm{h}M^2N_Q$, there exist $\bm{y}\in\Z^{\ell}$ and $\omega_{MN_Q\bm{r},\bm{k}}$ with $|\omega_{N_Q\bm{r},\bm{k}}|=1$ such that 
\[
\Phi\left(MN_Q\bm{r},\bm{k}\right)=\omega_{MN_Q\bm{r},\bm{k}}\sum_{\bm{h}\pmod{c}}e^{\frac{2\pi i}{c}\left(\frac{a}{2}\bm{h}^TM^2A_Q \bm{h} + \bm{y}^T M^2A_Q\bm{h}\right)}.
\]
Therefore, as in the proof of \cite[Theorem 4.2]{Blomer},
\[
\left|\Phi\left(N_Q\bm{r},\bm{k}\right)\right|^2  = \sum_{\bm{h_1},\bm{h_2}\pmod{c}}e^{\frac{2\pi i}{c}\left(\frac{a}{2}\left(\bm{h_1}^TM^2 A_Q \bm{h_1}-\bm{h_2}^TM^2 A_Q \bm{h_2}\right)+ \bm{y}^T M^2A_Q\left(\bm{h_1}-\bm{h_2}\right)\right)}.
\]
Making the change of variables $\bm{h}_3:=\bm{h}_1-\bm{h}_2$, this can be rewritten (using the bilinearity)
\begin{equation}\label{eqn:linearsum}
\sum_{\bm{h_3}\pmod{c}}e^{\frac{2\pi i}{c}\left(\frac{a}{2}\bm{h_3}^TM^2A_Q \bm{h_3}+ \bm{y}^T M^2A_Q\bm{h}_3\right)}\sum_{\bm{h_2}\pmod{c}}e^{\frac{2\pi i}{c}  a\bm{h_3}^TM^2A_Q\bm{h_2}}.
\end{equation}
We now make a change of variables to rewrite the matrix in Smith normal form; i.e., we have $A_Q=C_Q B_QD_Q$ with $C_Q$ and $D_Q$ having determinant $1$ and $B_Q$ being diagonal. Since $C_Q$ and $D_Q$ are invertible, the vectors $\bm{f}:=a^{-1}D_Q\bm{h_2}$ and $\bm{h}:=C_Q\bm{h_3}$ run over all integers modulo $c$ as well. Thus \eqref{eqn:linearsum} becomes
\begin{equation}\label{eqn:SmithNormal}
\sum_{\bm{h}\pmod{c}}e^{\frac{2\pi i}{c}\left(\frac{a}{2}\left(C_Q^{-1}\bm{h}\right)^TM^2A_Q C_Q^{-1}\bm{h}+ \bm{y}^T M^2A_QC_Q^{-1}\bm{h}\right)}\sum_{\bm{f}\pmod{c}}e^{\frac{2\pi i}{c}  \bm{f}^TM^2B_Q\bm{h}}.
\end{equation}
Since $B_Q$ is diagonal (say the diagonal entries are $b_1,b_2,\dots,b_{\ell}$), the inner sum in \eqref{eqn:SmithNormal} now becomes
\[
\prod_{j=1}^{\ell} \sum_{f_j\pmod{c}} e^{\frac{2\pi i}{c} f_j M^2 b_j h_j}.
\]
where the $j$-th component of $\bm{h}$ is denoted by $h_j$ and the $j$-th component of $\bm{f}$ is denoted by $f_j$. By orthogonality of roots of unity, the sum over $f_j$ vanishes unless $M^2b_jh_j\equiv 0\pmod{c}$. Thus the product is non-zero if and only if for every $j$
\[
\frac{c}{\gcd(M^2 b_j,c)}\mid h_j.
\]
For each such $h_j$, the sum is $c^{\ell}$ and there are $\gcd(M^2b_j,c)$ choices of $h_j$ modulo $c$ for which this occurs. Taking the absolute value around the inner sum in \eqref{eqn:SmithNormal}, we may bound 
\begin{align*}
|\Phi(N_Q\bm{r},\bm{k})|^2\leq \sum_{\substack{\bm{h}\pmod{c}\\ \frac{c}{\gcd(M^2b_j,c)}\mid h_j}} c^{\ell} &= c^{\ell}\prod_{j=1}^{\ell} \#\left\{h_j\pmod{c}: \frac{c}{\gcd(M^2b_j,c)}\mid h_j\right\}\\
&= c^{\ell}\prod_{j=1}^{\ell} \gcd(M^2b_j,c).
\end{align*}
So 
\[
|\Phi(MN_Q\bm{r},\bm{k})|\leq c^{\frac{\ell}{2}}\prod_{j=1}^{\ell} \gcd(M^2b_j,c)^{\frac{1}{2}}. 
\]
 By \cite[Proposition 2.1]{Shimura} (see also \cite[Chapter IX, (26)]{Schoeneberg}), $\theta(\tau;MN_Q\bm{r},M^2A_Q,M^2N_Q)$ is a weight $\frac{\ell}{2}$ modular form on $\Gamma_0(M^2N_Q)\cap\Gamma_1(M)$ with Nebentypus, so we take $N=M^2N_Q$ and $L=M$. In particular, we have $N=M^2N_Q$ in \eqref{eqn:adeltajdef}. Plugging back into \eqref{eqn:CuspExpansionOdd} and taking the difference between two quadratic forms, we see that 
\begin{equation}\label{eqn:CuspExpansionOddSimplify}
\left|a_{\delta,j}(n)\right|\leq 2\Delta_Q^{-\frac{1}{2}}M^{-\ell} \prod_{j=1}^{\ell} \gcd(M^2b_j,c)^{\frac{1}{2}} \sum_{\substack{\bm{k} \pmod{M^2N_Q}\\ A_Q\bm{k}\equiv 0\pmod{N_Q}}}r_{Q,\bm{k},MN_Q}(n),
\end{equation}
where $r_{Q,\bm{k},MN_Q}(n)$ is the number of solutions to $Q(\bm{x})=n$ with $\bm{x}\equiv \bm{k}\pmod{M^2N_Q}$.  We then note that by Lemma \ref{lem:Blomer} (with $M=1$)
\[
\sum_{\substack{\bm{k} \pmod{M^2N_Q}\\ A_Q\bm{k}\equiv 0\pmod{N_Q}}}r_{Q,\bm{k},M^2N_Q}(n)\leq R_{Q,\bm{1},1}(n)\leq \frac{3^{\ell}n^{\frac{\ell}{2}}}{\sqrt{\Delta_Q}} + \ell 3^{\ell-1} n^{\frac{\ell-1}{2}}\leq \sqrt{2}3^{\ell-1}n^{\frac{\ell-1}{2}}\left(\frac{9}{\Delta_Q} n +\ell^2\right)^{\frac{1}{2}}.
\]
Here we have used the inequality $(x+y)^2\leq 2x^2+2y^2$ in the last step. Thus \eqref{eqn:CuspExpansionOddSimplify} implies that 
\begin{equation}\label{eqn:CuspExpansionOddFinal}
\left|a_{\delta,j}(n)\right|^2\leq 8\Delta_Q^{-1}M^{-\ell} 3^{2\ell-2}n^{\ell-1}\left(\frac{9}{\Delta_Q} n +\ell^2\right)\prod_{d=1}^{\ell} \gcd(M^2b_d,c_j).
\end{equation}

Noting that $\ell\geq 4$ is even, we have (for example, see \cite[8.4.8]{NIST})
\begin{align*}
\left(\frac{N}{2\pi n}\right)^{\frac{\ell}{2}-2}\Gamma\left(\frac{\ell}{2}-1,\frac{2\pi n}{N}\right) &= \left(\frac{\ell}{2}-2\right)!e^{-\frac{2\pi n}{N}}\sum_{m=0}^{\frac{\ell}{2}-2}\frac{\left(\frac{2\pi n}{N}\right)^{m+2-\frac{\ell}{2}}}{m!}\\
&=\left(\frac{\ell}{2}-2\right)!e^{-\frac{2\pi n}{N}}\sum_{m=0}^{\frac{\ell}{2}-2}\frac{\left(\frac{2\pi n}{N}\right)^{-m}}{\left(\frac{\ell}{2}-2-m\right)!}.
\end{align*}
Plugging back into \eqref{eqn:innerexpand} yields
\begin{multline*}
\|f\|^2\leq \frac{3^{2\ell-2}\left(\frac{\ell}{2}-2\right)!}{2^{\frac{\ell}{2}-3}\pi} \frac{(M^2N_Q)^2}{\mu\left(M^2N_Q\right)}\sum_{\delta\mid M^2N_Q}\sum_{j=1}^{\varphi\left(\frac{M^2N_Q}{\delta}\right)\varphi(\delta)\frac{M^2N_Q}{\delta}}\frac{\prod_{d=1}^{\ell} \gcd(M^2b_d,c_j)}{M^{2\ell} \Delta_Q} \\
\times \sum_{m=0}^{\frac{\ell}{2}-2} \frac{\left(\frac{2\pi}{M^2N_Q}\right)^{-m}}{\left(\frac{\ell}{2}-2-m\right)!}\sum_{n=1}^{\infty} n^{\ell-2-m}e^{-\frac{2\pi n}{M^2N_Q}}\left(\frac{9}{\Delta_Q}n + \ell^2\right).
\end{multline*}
We now plug in Lemma \ref{lem:expsum} to bound the inner sum, yielding 
\begin{multline}\label{eqn:fQdNormbnd}
\|f\|^2\leq \frac{3^{2\ell-2}\left(\frac{\ell}{2}-2\right)!}{2^{\frac{\ell}{2}-3}\pi} \frac{M^4N_Q^2}{\mu(M^2N_Q)}\sum_{\delta\mid M^2N_Q}\sum_{j=1}^{\varphi\left(\frac{M^2N_Q}{\delta}\right)\varphi(\delta)\frac{M^2N_Q}{\delta}}\frac{\prod_{d=1}^{\ell} \gcd(M^2b_d,c_j)}{M^{2\ell}\Delta_Q} \\
\times\sum_{m=0}^{\frac{\ell}{2}-2} \frac{\left(\frac{2\pi}{M^2N_Q}\right)^{-m}}{\left(\frac{\ell}{2}-2-m\right)!}\left(\frac{9}{ \Delta_Q} \frac{(\ell-m-1)!}{\left(\pi\delta_{M^2N_Q}\right)^{\ell-m}} (M^2N_Q)^{\ell-m} +\ell^2 \frac{(\ell-m-2)!}{\left(\pi\delta_{M^2N_Q}\right)^{\ell-m-1}} (M^2N_Q)^{\ell-m-1}\right).
\end{multline}
Noting that 
\[
\prod_{d=1}^{\ell}b_d=\det(B_Q)=\det(A_Q)=\Delta_Q,
\]
we may trivially bound 
\[
\prod_{j=1}^{\ell}\gcd(M^2b_d,c_j) \leq \gcd(M^2,c_j)^{\ell}\prod_{j=1}^{\ell}b_d= \gcd(M^2,c_j)^{\ell}\Delta_Q.
\]
Moreover, recalling that the sum over $j$ runs oer those cusps with $\gcd(c_j,M^2N_Q)=\delta$, we have $\gcd(M^2,c_j)=\gcd(M^2,\delta)$. Thus \eqref{eqn:fQdNormbnd} implies that 
\begin{multline}\label{eqn:fQdNormbndfull}
\|f\|^2\leq \frac{3^{2\ell-2}\left(\frac{\ell}{2}-2\right)!}{2^{\frac{\ell}{2}-3}\pi^{\ell}\delta_{M^2N_Q}^{\ell-1}} \frac{M^{2\ell+2}N_Q^{\ell+1}}{\mu(M^2N_Q)}\sum_{\delta\mid M^2N_Q}\varphi\left(\frac{M^2N_Q}{\delta}\right)\varphi(\delta)\frac{M^2N_Q}{\delta}\left(\frac{\gcd(M^2,\delta)}{M^2}\right)^{\ell}\\
\times  \sum_{m=0}^{\frac{\ell}{2}-2} \frac{\left(\frac{\delta_{M^2N_Q}}{2}\right)^{m}}{\left(\frac{\ell}{2}-2-m\right)!}(\ell-m-2)!\left(\frac{9}{ \Delta_Q} (\ell-m-1)\frac{ M^2N_Q}{\pi \delta_{M^2N_Q}} +\ell^2  \right).
\end{multline}
Plugging \eqref{eqn:Gammaindex} into \eqref{eqn:fQdNormbndfull} yields
\begin{multline}\label{eqn:fQdNormbndfinalsum}
\|f\|^2\leq \frac{3^{2\ell-2}\left(\frac{\ell}{2}-2\right)!}{2^{\frac{\ell}{2}-3}\pi^{\ell}\delta_{M^2N_Q}^{\ell-1}} \frac{M^{2\ell-4}N_Q^{\ell-2}}{\prod_{p\mid M^2N_Q}\left(1-p^{-2}\right)}\sum_{\delta\mid M^2N_Q}\varphi\left(\frac{M^2N_Q}{\delta}\right)\varphi(\delta)\frac{M^2N_Q}{\delta}\left(\frac{\gcd(M^2,\delta)}{M^2}\right)^{\ell} \\
\times \sum_{m=0}^{\frac{\ell}{2}-2} \frac{(2\pi)^{-m}}{\left(\frac{\ell}{2}-2-m\right)!}(\ell-m-2)! \left(\frac{9}{ \Delta_Q}(\ell-m-1)\frac{M^2N_Q}{\pi\delta_{M^2N_Q}} +\ell^2 \right).
\end{multline}
Noting that $M$, $N_Q$, and $\Delta_Q$ are all invariants of the genus, the bound in \eqref{eqn:fQdNormbndfinalsum} is independent of the choice of the shifted lattice $L'+\frac{v'}{M}$ occuring in the definition \eqref{eqn:fcuspdef} of $f$. Plugging \eqref{eqn:fQdNormbndfinalsum} into \eqref{eqn:normSiegelWeil} thus yields 
\begin{multline}\label{eqn:normbound}
\|f_{r,M,\bma}\|^2\leq \frac{3^{2\ell-2}\left(\frac{\ell}{2}-2\right)!}{2^{\frac{\ell}{2}-3}\pi^{\ell}\delta_{M^2N_{\bma}}^{\ell-1}} \frac{M^{2\ell-4}N_{\bma}^{\ell-2}}{\prod_{p\mid M^2N_{\bma}}\left(1-p^{-2}\right)}\sum_{\delta\mid M^2N_{\bma}}\varphi\left(\frac{M^2N_{\bma}}{\delta}\right)\varphi(\delta)\frac{M^2N_{\bma}}{\delta}\left(\frac{\gcd(M^2,\delta)}{M^2}\right)^{\ell} \\
\times \sum_{m=0}^{\frac{\ell}{2}-2} \frac{(2\pi)^{-m}}{\left(\frac{\ell}{2}-2-m\right)!}(\ell-m-2)! \left(\frac{9}{ \Delta_{\bma}}(\ell-m-1)\frac{M^2N_{\bma}}{\pi\delta_{M^2N_{\bma}}} +\ell^2 \right).
\end{multline}
Plugging back into Lemma \ref{lem:cf(n)<norm}, we obtain \eqref{eqn:afQbndgen}, giving us part (1).

After noting that the sum collapses to a single term in the special case $\ell=4$, part (2) now follows immediately from Lemma \ref{lem:sigma0N}.
\end{proof}

\section{Eisenstein series} \label{sec:Eisenstein}

In this section, we use \eqref{eqn:localdensity} to investigate the coefficients $A_{r,M,\bma}(n)$ of the Eisenstein series part of the decomposition \eqref{eqn:decomposition}. The local densities have been computed in a number of places. In the case of lattices (i.e., for quadratic forms without congruence conditions), many cases were computed by Siegel \cite{SiegelQuadForm} and Weil \cite{Weil} and a full computation was done by Yang \cite{Yang}. Shimura \cite{Shimura} computed the local densities for shifted lattices (i.e., quadratic forms with congruence conditions).

We use the results from \cite[Theorem 1.5]{ShimuraCongruence} here, specializing to the rank 4 case. Let 
\[
Q(\bm{x}):=Q_{\bma}(\bm{x})=\sum_{j=1}^4 a_jx_j^2
\]
and suppose that $L$ is the associated lattice coming from the standard basis. To translate into Shimura's notation, we fix a sublattice $\mathfrak{L}:=M L$ and a vector $\bm{\nu}$ given by $\nu_j:=r$ and then define the characteristic function 
\[
\lambda(\bm{x}):=\begin{cases}1 &\text{if }\bm{x}\in \mathfrak{L}+\bm{\nu},\\ 0&\text{otherwise}.\end{cases}
\]
Analogously to \eqref{eqn:SiegelWeil}, Shimura then defines an average over the genus, which he denotes by $R(h,\lambda)$, and a mass $m(\lambda)$ of the genus, which is the sum appearing in the denominator of \eqref{eqn:SiegelWeil}. In this notation, we have 
\[
\frac{R(h,\lambda)}{m(\lambda)}=A_{r,M,\bma}(h)
\]
and we set $h:=4M n+(r-4)^2\sum_{j=1}^4 a_j$. We are particularly interested in the case $r=7$ and $M=10$, so $h=40n+9\sum_{j=1}^4 a_j$ in our case.

The field $F$ appearing in \cite[Theorem 1.5]{ShimuraCongruence} is $F=\Q$ in our case. We let $e_1=e_1(Q,M)$ denote the product of all primes such that $p|e_1$ if and only if $\mathfrak{L}_p=M L_p$ is not maximal and $e'=e'(Q,M)$ denote the product of all primes dividing the discriminant $\Delta_Q$. For $x\in \Q_p$ and $y\in \bigcup_{j=1}^{\infty} p^{-j}\Z$ such that $x-y\in\Z_p$, we set $\bm{e}_p(x):=e^{-2\pi i y}$.  Moreover, we denote by $\psi$ the Hecke ideal character of $\Q$ corresponding to $K/\Q$, where $K:=\Q(\sqrt{-\Delta_Q})$, and for $p\nmid e_1$ we set 
\[
\gamma_p(2):=\frac{\beta_{Q,p}(h)}{1-\psi(p)p^{-2}},
\]
while for $p\mid e_1$ we denote 
\[
b_p(h,\lambda,0):=\int_{\Q_p}\int_{V_p} \bm{e}_p(\sigma(Q(\bm{v})-h))\lambda(\bm{v}) d\bm{v} d\sigma =\beta_{Q,p,r,M}(h).
\]
Using this notation, Shimura rewrites \eqref{eqn:localdensity} as
\begin{equation}\label{eqn:shimuras}
A_{r,M,\bma}(h)=\frac{R(h,\lambda)}{m(\lambda)} = \frac{(2\pi)^{2}h}{[\widetilde{L}:L]^{1/2}L(2,\psi)}\cdot\prod_{p|e_1}\frac{b_{p}(h,\lambda,0)}{(1-\psi(p)p^{-2})}\cdot\prod_{p|he',p\nmid e_1}\gamma_{p}(2).
\end{equation}
Recalling the definition of $\Delta_{\bma}$ from \eqref{eqn:Deltaadef} and that $\widetilde{L}=2L^*$, we have 
\[
\left[\widetilde{L}:L\right]=\Delta_Q=\Delta_{\bma}.
\]
For ease of notation, we hence set 
\[
C_{L,\bma,M}:=\frac{(2\pi)^{2}}{\Delta_{\bma}^{1/2}L(2,\psi)}\cdot\prod_{p\mid e_1}\frac{1}{(1-\psi(p)p^{-2})}.
\]
In this paper, we fix $M:=2(m-2)=10$ and abbreviate $C_{L,\bma}:=C_{L,\bma,10}$. To illustrate this method to bound $A_{7,10,\bma}\left(40n+9\sum_{j=1}^4\aj_j\right)$, we consider the example of $\bma=(1,1,1,3)$. We would like to show that there exists $c_{(1,1,1,3)}$ such that
\[
A_{7,10,(1,1,1,3)}(40n+54)\geq c_{(1,1,1,3)}(40n+54)^{1-\delta}
\]
for some small constant $\delta$. In the following lemma, we compute $c_{\bma}$ for $\bma = (1,1,1,3)$ and in the next section we will follow the same method to compute $c_{\bma}$ for all $\bma$.
\begin{lemma}\label{lem:lowerA1113}
$c_{(1,1,1,3)}=\frac{1}{1300}$ for $\delta=10^{-6}$. That is, we have
\[
A_{7,10,(1,1,1,3)}^*(40n+54)\geq \frac{1}{1300}(40n+54)^{1-10^{-6}}
\]
\end{lemma}
\begin{proof}
We use equation \eqref{eqn:shimuras} to compute $A_{7,10,(1,1,1,3)}(40n+54)$ and bound it from below. 
\begin{equation}\label{eqn:A7101113}
A_{7,10,(1,1,1,3)}(40n+54) = C_{L,\bm{a}} h \prod_{p|10}b_{p}(h,\lambda,0)\cdot\prod_{p|e_1,p\nmid 10}b_{p}(h,\lambda,0)\cdot\prod_{p|he',p\nmid e_1}\gamma_{p}(2).
\end{equation}
To evaluate $b_p(h,\lambda,0)$ for $p\mid M$, we follow the proof in \cite[Theorem 3.4 and Theorem 3.8]{KL} with $N=M=10$ and $h\mapsto M^2 h$. For odd $p$, we have 
\begin{align*}
b_p(h,\lambda,0)&=\int_{\Q_p}\int_{\mathfrak{L}_p} \bm{e}_p\left(\sigma \left( Q(\bm{v}+\bm{\nu})-h\right)\right)d\bm{v} d\sigma\\
&=\int_{\Q_p}\bm{e}_p\left(\sigma\left(\sum_{j=1}^4a_j \nu_j^2-h\right)\right)\frac{1}{p^{4\ord_p(M)}} \prod_{j=1}^{4}\int_{\Z_p}\bm{e}_p\left( M \sigma a_j \left(M\alpha^2+2\alpha \nu_j\right)\right) d\alpha d\sigma
\end{align*}
By \cite[Lemma 3.2]{KL}, the inner integral is zero unless $M\sigma a_j\in \Z_p$, in which case it equals $1$. Since $a_1=1$ and $\sum_{j=1}^4 a_j \nu_j^2\equiv h\pmod{p^{\ord_p(M)}}$, the integral simplifies as 
\[
p^{-4\ord_p(M)} \int_{M^{-1}\Z_p} d\sigma=p^{-3\ord_p(M)}.
\]
By the analogous calculation for $p=2$, using \cite[Lemma 3.6]{KL}, we conclude that 
\begin{equation}\label{eqn:condbnd}
\prod_{p|10}b_{p}(h,\lambda,0)=b_{2}(h,\lambda,0)\cdot b_{5}(h,\lambda,0)=\frac{1}{250}.
\end{equation}
For the primes $p\mid \Delta_Q=:\Delta$ with $p\nmid 10$, we bound the product of local densities from below by a constant $C_{\Delta,\bma}$, using the precise computation depending on $\bma$ (see Table \ref{tab:bounds} in Appendix \ref{sec:appendix}). Namely, we have 
\begin{equation}\label{eqn:Cbnd}
\prod_{p|e_1,p\nmid 10}b_{p}(h,\lambda,0)\cdot\prod_{p\nmid h,p|e',p\nmid e_1}\gamma_p(2)\geq C_{\Delta,\bma}>0.
\end{equation}
For the remaining primes, a straightforward calculation yields (for details of the last inequality, see for example \cite[Lemma 2.2]{BanerjeeKane}) 
\begin{equation}\label{eqn:locdendivh}
\prod_{p|h,p\nmid e_1e'}\gamma_{p}(2) \geq \prod_{p|h}\left(1-p^{-1}\right) \geq \frac{1}{20}h^{-10^{-6}}.
\end{equation}
Recalling that 
\[
L(2,\psi)=\frac{\zeta_K(2)}{\zeta(2)},
\]
in order to evaluate $L(2,\psi)$ we use code of Brandon Williams (based on the proof of the Klingen--Siegel Theorem; see \cite{Klingen,SiegelQuadFormIII}), to compute the Dedekind zeta value $\zeta_K(2)$. One may alternatively use well-known identities by computing $\zeta_{K}(-1)$ and using the functional equation. In this case, we have $C_{L,\bma}=\frac{75}{13}$, as tabulated in  Table \ref{tab:bounds}. Hence, plugging \eqref{eqn:condbnd}, \eqref{eqn:Cbnd}, and \eqref{eqn:locdendivh} into \eqref{eqn:A7101113}, we obtain
\[
A_{7,10,(1,1,1,3)}(40n+54) \geq  \frac{3}{2600} C_{\Delta,\bma} h^{1-10^{-6}}.
\]
The bound $C_{\Delta,\bma}$ follows by a standard calculation. We will give details exactly the way they are presented in \cite[Lemma 4.6]{BanerjeeKane}. In the current case under consideration, we have only one prime to consider, $p=3$. To compare with the notation in \cite{BanerjeeKane}, we define
\begin{equation}\label{eqn:alphadef}
\bm{\alpha}:=\left(\ord_3\left(\aj_{\rho(1)}\right),\ord_3\left(\aj_{\rho(2)}\right),\ord_3\left(\aj_{\rho(3)}\right),\ord_3\left(\aj_{\rho(4)}\right)\right),
\end{equation}
where $\rho$ is a permutation of $\{1,2,3,4\}$ so that $\bm{\alpha}$ is in increasing order. We see that in our case we have $\bm{\alpha}=(0,0,0,1)$,

 Letting $R:=\ord_3(h)$, there are two possible cases. Either $R=0$, in which case we are in case \cite[Lemma 4.6 (5d)]{BanerjeeKane}, or $R\geq 1$, in which case we are in case \cite[Lemma 4.6 (5e)]{BanerjeeKane}.

If $R=0$, then the second and third terms vanish because $\alpha_2=\alpha_3=R=0$, so we have ($\mathcal{A}_j:=\sum_{k\leq j} \alpha_k$)
\[
\beta_{Q,3}(40n+54)=1 + \eta_{3,3}^{\mathcal{A},\operatorname{e}} 3^{-1}\chi_3(-m')\geq 1-3^{-1}=\frac{2}{3}.
\]
Here $\eta_{p,3}^{\mathcal{A},\operatorname{e}}$ and $\chi_p(-m')$ are certain roots of unity.

If $R\geq 1$, then by \cite[Lemma 4.6 (5)]{BanerjeeKane}, the second and third terms again vanish (note that $\left\lfloor\frac{\alpha_4}{2}\right\rfloor =0$) and we obtain 
\[
\beta_{Q,3}(40n+54)= 1 + \eta_{3,4}^{\mathcal{A},\operatorname{e}} 3^{-R-1}> 1-3^{-1}.
\]
In both cases, we obtain the lower bound $C_{\Delta,\bma} = \frac{2}{3}$. Thus,
\[
A_{7,10,(1,1,1,3)}(40n+54) \geq \frac{3}{2600}(40n+54)^{1-10^{-6}}\cdot\frac{2}{3} = \frac{1}{1300}(40n+54)^{1-10^{-6}},
\]
which completes the proof.
\end{proof}

Following the proof of Lemma \ref{lem:lowerA1113}, we may generalize the result for a given choice of $\bma$. In order to prove Theorem \ref{thm:gamma7bound}, we require bounds for the cases of $\bma$ appearing in Appendix \ref{sec:appendix}. To state the result, we define 
\begin{equation}\label{eqn:ca}
c_{\bma} := \frac{C_{L,\bma}C_{\Delta,\bma}}{5000}, 
\end{equation}
whose value may be found in each case in Table \ref{tab:bounds}.
\begin{lemma}\label{lem:lowerAgen}
For any choice of $\bma$ appearing in Table \ref{tab:bounds} we have
\begin{equation}\label{eqn:lowerAgen}
A_{7,10,\bma}\left(40n+9\sum_{j=1}^4\aj_j\right) \geq c_{\bma} h^{1-10^{-6}}.
\end{equation}

\end{lemma}
\begin{proof}
One can directly use \eqref{eqn:condbnd} and \eqref{eqn:locdendivh} for any choice of $\bma$, so the proof of Lemma \ref{lem:lowerA1113} generalizes by replacing the constants $C_{L,\bma}$ and $C_{\Delta,\bma}$ with the corresponding values given in Table \ref{tab:bounds}. Therefore, for any $\bma$, we have 
\[
A_{7,10,\bma}\left(40n+9\sum_{j=1}^4\aj_j\right) \geq \frac{1}{250} \cdot \frac{1}{20} \cdot h^{1-10^{-6}} \cdot C_{L,\bma} \cdot C_{\Delta,\bma} = \frac{1}{5000}\cdot h^{1-10^{-6}} \cdot C_{L,\bma} \cdot C_{\Delta,\bma}= c_{\bma} h^{1-10^{-6}},
\]
 yielding \eqref{eqn:lowerAgen}.

\end{proof}

\section{Putting it all together: proof of Theorem \ref{thm:gamma7bound}}\label{sec:mainproof}

We next get a uniform bound on $B_{7,10,\bma}^*\left(40n+9\sum_{j=1}^4 \aj_j\right)$ for every case under consideration. 

\begin{lemma}\label{lem:Bboundgen}
Suppose that 
\begin{multline*}
\bma\in \{(1,1,1,k):k\leq 10\}\cup \{(1,1,2,k):2\leq k\leq 23\}\cup \{(1,1,3,k):3\leq k\leq 6\}\\
\cup \{(1,2,2,k):2\leq k\leq 19\}\cup \{(1,2,3,k):3\leq k\leq 31\}\\
\cup \{(1,2,4,k):4\leq k\leq 131\}\cup \{(1,2,5,k):5\leq k\leq 10\}.
\end{multline*}
Then 
\[
\left|B_{7,10,\bma}\left(40n+9\sum_{j=1}^4\aj_j\right)\right|\leq 5.84\cdot 10^{38} \left(40n+9\sum_{j=1}^4\aj_j\right)^{\frac{3}{5}}.
\]

\end{lemma}
\begin{proof}
We first note that $\Theta_{7,10,\bma}$ is a modular form of weight two on $\Gamma_0(400\lcm(\bma))\cap\Gamma_1(10)$ with some character. By \cite[(4.4)]{BanerjeeKane}, we have 
\begin{multline}\label{eqn:genBbound}
\left|B_{7,10,\bma}\left(40n+9\sum_{j=1}^4 \aj_j\right)\right|\leq 6.95\cdot 10^{18} \left(40n+9\sum_{j=1}^4\aj_j\right)^{\frac{3}{5}}\\
\times \|f_{7,10,\bma}\| \Big(400\lcm(\bma)\Big)^{1+2.5\cdot 10^{-6}}\prod_{p\mid 20\lcm(\bma)} \left(1+\frac{1}{p}\right)^{\frac{1}{2}}\cdot 4.
\end{multline}
Given the choices of $\bma$, we may uniformly bound $\lcm(\bma)\leq 524$ (attained for $\bma=(1,2,4,131)$ and 
\[
\prod_{p\mid 20\lcm{\bma}}\left(1+\frac{1}{p}\right)\leq \frac{96}{35},
\]
(attained for $\bma=(1,2,4,105)$). Plugging back into \eqref{eqn:genBbound} yields
\begin{equation}\label{eqn:Bboundnorm}
\left|B_{7,10,\bma}\left(40n+9\sum_{j=1}^4 \aj_j\right)\right|\leq 9.66\cdot 10^{24} \left(40n+9\sum_{j=1}^4\aj_j\right)^{\frac{3}{5}}\|f_{7,10,\bma}\|.
\end{equation}

By Lemma \ref{lem:cuspbound} (and \eqref{eqn:normbound} in particular), we have 
\begin{align*}
\|f_{7,10,\bma}\|^2&\leq \frac{4\cdot 3^{10}}{\pi^{4}} \frac{10^4\cdot \big(4\lcm(\bma)\big)^2}{\prod_{p\mid 10\lcm(\bma)} \left(1-p^{-2}\right)}\left(\frac{27}{16\prod_{j=1}^4\aj_j}\frac{400\lcm(\bma)}{\pi}+16\right)\\
&\quad \times \sum_{\delta\mid 400\lcm(\bma)}\varphi\left(\frac{400\lcm(\bma)}{\delta}\right)\varphi(\delta)\frac{400\lcm(\bma)}{\delta}\left(\frac{\gcd(100,\delta)}{100}\right)^4.
\end{align*}
We bound $\lcm(\bma)\leq 524$, $\prod_{j=1}^4\aj_j\geq 1$, and 
\[
\prod_{p\mid 10\lcm(\bma)} \left(1-p^{-2}\right)\geq \frac{768}{1225}
\]
(attained for $\bma=(1,2,4,105)$) to obtain  
\begin{equation}\label{eqn:sumleftcuspidal}
\|f_{7,10,\bma}\|^2\leq 6.97\cdot 10^{13} \cdot \lcm(\bma)^2\sum_{\delta\mid 400\lcm(\bma)}\varphi\left(\frac{400\lcm(\bma)}{\delta}\right)\varphi(\delta)\frac{400\lcm(\bma)}{\delta}\left(\frac{\gcd(100,\delta)}{100}\right)^4.
\end{equation}
Using the fact that $\lcm(\bma)\leq 524$, a quick computer check shows that 
\begin{multline*}
\lcm(\bma)^2\sum_{\delta\mid 400\lcm(\bma)}\varphi\left(\frac{400\lcm(\bma)}{\delta}\right)\varphi(\delta)\frac{400\lcm(\bma)}{\delta}\left(\frac{\gcd(100,\delta)}{100}\right)^4\\
\leq \max_{L\leq 524} L^2\sum_{\delta\mid 400L}\varphi\left(\frac{400L}{\delta}\right)\varphi(\delta)\frac{400L}{\delta}\left(\frac{\gcd(100,\delta)}{100}\right)^4=\frac{32721047140294656}{625},
\end{multline*}
which is indeed attained for $\bma=(1,2,4,131)$. Therefore 
\[
\|f_{7,10,\bma}\|^2\leq 3.65\times 10^{27}
\]
and plugging this into \eqref{eqn:Bboundnorm} yields
\[
\left|B_{7,10,\bma}\left(40n+9\sum_{j=1}^4\aj_j\right)\right|\leq 5.84\cdot 10^{38} \left(40n+9\sum_{j=1}^4\aj_j\right)^{\frac{3}{5}}.\qedhere
\]
\end{proof}

For comparison, in the case $\bma=(1,2,4,131)$, plugging in the explicit evaluation for each term in \eqref{eqn:genBbound} yields
\[
\frac{\left|B_{7,10,(1,2,4,131)}\left(40n+1242\right)\right|}{\left(40n+1242\right)^{\frac{3}{5}}}\leq 7.85\cdot 10^{24}\|f_{7,10,\bma}\|.
\]
The bound on the norm is 
\begin{align*}
\|f_{7,10,(1,2,4,131)}\|^2&\leq \frac{4\cdot 3^{10}}{\pi^{4}} \frac{10^4\cdot 2096^2}{\prod_{p\mid 1310} \left(1-p^{-2}\right)}\left(\frac{27}{16\cdot 1024}\frac{400\cdot 524}{\pi}+16\right)\\
&\quad \times \sum_{\delta\mid 400\cdot 524}\varphi\left(\frac{400\cdot 524}{\delta}\right)\varphi(\delta)\frac{400\cdot 524}{\delta}\left(\frac{\gcd(100,\delta)}{100}\right)^4\leq 4.90\cdot 10^{25}.
\end{align*}
Combining these together yields 
\[
B_{7,10,(1,2,4,131)}(40n+1242)\leq 5.50\times 10^{37}\left(40n+1242\right)^{\frac{3}{5}}.
\]




Now we have all the pieces we need to explicitly compute the bound $C_{\bma}$ for any $\bma$ in Table \ref{tab:bounds} such that every integer larger than $C_{\bma}$ is represented as in \eqref{eqn:sumsmgonal}.
\begin{proposition}\label{prop:Ca}
Suppose that $\bma$ is one of the choices appearing in Table \ref{tab:bounds}. Then if 
\[
n>C_{\bma}:=\frac{\left(\frac{d_{\bma}}{c_{\bma}}\right)^{\frac{1}{\left(2/5-10^{-6}\right)}} - \left(9\sum_{j=1}^4\aj_j\right)}{40},
\]
 where $d_{\bma} = 5.84\cdot 10^{38}$ and $c_{\bma} = \frac{C_{L,\bma}\cdot C_{\Delta,\bma}}{5000}$, then 
\[
a_1p_7(x_1)+a_2p_7(x_2)+a_3p_7(x_3)+a_4p_7(x_4) =n
\]
is solvable with $\bm{x}\in\Z^4$. 
\end{proposition}
\begin{proof}
We have
\[
A_{7,10,\bma}\left(40n+9\sum_{j=1}^4\aj_j\right) \geq c_{\bma}\left(40n+9\sum_{j=1}^4\aj_j\right)^{1-10^{-6}}
\]
and 
\[
\left|B_{7,10,\bma}\left(40n+9\sum_{j=1}^4\aj_j\right)\right|\leq 5.84\cdot 10^{38} \left(40n+9\sum_{j=1}^4\aj_j\right)^{\frac{3}{5}}.
\]
Suppose that $s_{7,10,\bma}\left(40n+9\sum_{j=1}^4\aj_j\right)=0$. Then 
\[
A_{7,10,\bma}\left(40n+9\sum_{j=1}^4\aj_j\right)=-B_{7,10,\bma}\left(40n+9\sum_{j=1}^4\aj_j\right).
\]

Therefore, since $A_{7,10,\bma}\left(40n+9\sum_{j=1}^4\aj_j\right)=\left|B_{7,10,\bma}\left(40n+9\sum_{j=1}^4\aj_j\right)\right|$, we have 
\[
c_{\bma} \left(40n+9\sum_{j=1}^4\aj_j\right)^{1-10^{-6}}\leq  d_{\bma} \left(40n+9\sum_{j=1}^4\aj_j\right)^{\frac{3}{5}},
\]
which in turn implies that 
\[
\left(40n+9\sum_{j=1}^4\aj_j\right)^{\frac{2}{5}-10^{-6}}\leq \frac{d_{\bma}}{c_{\bma}}.
\]
Therefore, 
\[
40n+9\sum_{j=1}^4\aj_j\leq \left(\frac{d_{\bma}}{c_{\bma}}\right)^{\frac{1}{\left(2/5-10^{-6}\right)}}.
\]
So $n>C_{\bma}$, as claimed.
\end{proof}

In addition to the choices of $\bma$ appearing in the Table \ref{tab:bounds}, we require one additional case, namely $\bma=(1,1,3,3)$. For this, we require the following lower bound.
\begin{lemma}\label{lem:a1133}
For $\bma=(1,1,3,3)$, we have
\[
A_{7,10,\bma}(40n+72) \geq \frac{\left( 3^{-R} (40n+72)\right)^{1-10^{-6}}}{2400}
\]
where $R = ord_3(40n+72)$.
\end{lemma}
\begin{remark}
Since $3^{-R}$ appears in the numerator, $A_{7,10,\bma}(40n+72)$ does not grow as the $3$-power goes to infinity, which is the reason that we must deal with this case separately. This behaviour occurs because the prime $3$ is what is known as an \begin{it}anisotropic prime\end{it}, which means that the $p$-adic lattice has no non-trivial isotropic elements (an element $\bm{x}\in \Z_p^4$ is isotropic if $Q(\bm{x})=0$).
\end{remark}
\begin{proof}
We refer to case (6) of \cite[Lemma 4.6]{BanerjeeKane} to complete the proof. For $\bma=(1,1,3,3)$ and $p=3$, we have $\bm{\alpha}=(0,0,1,1)$, as defined in \eqref{eqn:alphadef}. Let $h=40n+72=3^R\cdot h'$, such that $3\nmid h'$. If $R=0$, by \cite[Lemma 4.6 (6)(a)]{BanerjeeKane} we have
\[
\beta_{Q_{(1,1,3,3)},3}(m)=1-\eta^{\mathcal{A},e}_{3,2}3^{-1} \geq 1-3^{-1}
\]
If $R \geq 1$, by \cite[Lemma 4.6 (6)(b)]{BanerjeeKane} we have
\[
\beta_{Q_{(1,1,3,3)},3}(m) \geq 3^{-1} - \frac{3^{-1}(1-3^{-2\lfloor{\frac{R}{2}}\rfloor})+3^{-2}(1-3^{-2\lfloor{\frac{(R-1)}{2}}\rfloor})}{1+3^{-1}} + 3^{-R-1}
\]
Simplifying it further, we get
\[
\beta_{Q_{(1,1,3,3)},3}(m) \geq \frac{3^{-1-2\lfloor{\frac{R}{2}}\rfloor}+3^{-2\lfloor{\frac{(R-1)}{2}}\rfloor} + 3^{-R-1}+ 3^{-R-2}}{1+3^{-1}}
\]
If $R$ is odd, then this simplifies as
\[
\beta_{Q_{(1,1,3,3)},3}(m) \geq \frac{3^{-R}+2\cdot 3^{-R-1}+ 3^{-R-2}}{1+3^{-1}} = \frac{4}{3}\cdot 3^{-R}.
\]
If $R$ is even, then we have
\[
\beta_{Q_{(1,1,3,3)},3}(m) \geq \frac{2\cdot 3^{-R-1}+3^{-R}+ 3^{-R-2}}{1+3^{-1}} = \frac{4}{3}\cdot 3^{-R}.
\]
Therefore, for all $R \geq 1$,
\[
\beta_{Q_{(1,1,3,3)},3}(m) \geq \frac{4}{3}\cdot 3^{-R}
\]
We additionally note that,
\[
\frac{\beta_{Q_{(1,1,3,3},3}(m)}{1-\psi(p)p^{-2}} \geq \frac{4}{3}\cdot 3^{-R}\cdot \frac{9}{10} = \frac{6}{5}\cdot 3^{-R} \geq C_{\Delta,\bma}
\]
Hence, $C_{\Delta,\bma}$ for $\bma=(1,1,3,3)$ is not independent of $h$, as it is in other cases. Replacing $h$ with $h'$ in \eqref{eqn:locdendivh} and using a computer to explicitly calculate $C_{L,\bma} = \frac{25}{9}$, we get
\[
A_{7,10,(1,1,1,3)}(40n+72) \geq \frac{C_{L,\bma}C_{\Delta,\bma}}{5000} 3^R h'^{1-10^{-6}} = \frac{3^RC_{\Delta,\bma}}{1800}h'^{1-10^{-6}} \geq \frac{h'^{1-10^{-6}}}{1500}.
\]
Since $h'=3^{-R}h$, the claim follows.
\end{proof}
We are now ready to prove Theorem \ref{thm:gamma7bound}.

\begin{proof}[Proof of Theorem \ref{thm:gamma7bound}]
We use a trick of Ramanujan \cite{RamanujanQuaternary} when determining the universal sums of four repeated squares, which was later greatly generalized by Bhargava \cite{Bhargava}
to the escalator tree method. By ordering the $\bma$, we may assume without loss of generality that $a_1\leq a_2\leq \dots\leq a_{\ell}$. If $\bma$ is universal, then it must represent $1$, and one easily sees that $a_1=1$. By looking for the next missing integer (called the \begin{it}truant\end{it}) and proceeding through all possible choices, one obtains a tree (called the escalator tree). The choices for $(a_1,a_2,a_3,a_4)$ at depth four are precisely $(1,1,1,k)$ with $k\leq 10$, $(1,1,2,k)$ with $2\leq k\leq 23$, $(1,1,3,k)$, $3\leq k\leq 6$, $(1,2,2,k)$ with $2\leq k\leq 19$, $(1,2,3,k)$ with $3\leq k\leq 31$, $(1,2,4,k)$ with $4\leq k\leq 131$, and $(1,2,5,k)$ with $5\leq k\leq 10$. 

For each of these choices of $\bma$ other than $\bma=(1,1,3,3)$, Proposition \ref{prop:Ca} implies that \eqref{eqn:sumsmgonal} is solvable for $n>C_{\bma}$. The largest value of $C_{\bma}$ is $3.896\cdot 10^{106}$, which occurs in the case $\bma=(1,2,4,108)$.

For $\bma=(1,1,3,3)$, we first note that 
\[
9=p_7(x_1)+p_7(x_2)+3p_7(x_3)+3p_7(x_4)
\]
is not solvable with $\bm{x}\in\Z^4$. We thus have $3\leq a_5\leq 9$. By Lemma \ref{lem:a1133}, if $R\leq 2$, then we have 
\[
A_{7,10,(1,1,3,3)}(40n+72)\geq \frac{1}{9^{1-10^{-6}} 1500} (40n+72)^{1-10^{-6}},
\]
while Lemma \ref{lem:Bboundgen} implies that 
\[
B_{7,10,(1,1,3,3)}(40n+72)\leq 5.84\cdot 10^{38} (40n+72)^{\frac{3}{5}}.
\]
Hence if $s_{7,10,(1,1,3,3)}(40n+72)=0$, then 
\[
 (40n+72)^{\frac{2}{5}-10^{-6}}\leq 5.84\cdot 10^{38}\left(9^{1-10^{-6}} 1500\right),
\]
so $n<4.37\cdot 10^{105}$. Thus if $n\geq 4.37\cdot 10^{105}$ and $27\nmid 40n+72$, then \eqref{eqn:sumsmgonal} is solvable.  

Now note that if $27\mid 40n+72$, then for any $3\leq k\leq 9$ we have 
\[
27\nmid 40(n-k)+72.
\]
Therefore, taking $x_5=1$, we see that every $n>4.37\cdot 10^{105}+k$ is represented as in \eqref{eqn:sumsmgonal} with the first 5 variables.

We have hence concluded in every case that for every $n\geq 3.896\cdot 10^{106}$ is represented by the first 4 or 5 variables in $\bm{a}$. The possible truants for the rest of the tree (which is finite) must all hence be a subset of $n< 3.896\cdot 10^{106}$. Thus if a given $\bma$ represents every integer up to $3.896\cdot 10^{106}$, then it is universal, yielding the claim.
\end{proof}

\appendix
\section{Data from explicit computations}\label{sec:appendix}
The relevant data for the proof of Theorem \ref{thm:gamma7bound} is collected in Table \ref{tab:bounds} in this appendix. We compute $C_{L,\bma}$ directly using a computer. To compute $C_{\Delta,\bma}$, we follow \cite[Lemma 4.6]{BanerjeeKane}, by carefully computing $\bm{\alpha}$ for each $\bma$ and bounding it from below, similar to the argument shown in Lemma \ref{lem:lowerA1113}. In most cases, the lower bound is simply $1-p^{-1}$ for prime p. More importantly, the lower bound is always positive and independent of $n$ for all $\bma$, except for $\bma = (1,1,3,3)$. We tabulate the computations in the table below for all $\bma$, except for $\bma = (1,1,3,3)$. $c_{\bma}$ is computed using equation \eqref{eqn:ca}, while $C_{\bma}$ is computed using the result of Proposition \ref{prop:Ca}.

\begin{longtable}[c]{|c|c|c|c|c||c|c|c|c|c|}
\hline
\tiny $\bma$ & \tiny $C_{L,\bma}$ & \tiny $C_{\Delta,\bma}$ & \tiny $c_{\bma}$ & \tiny $C_{\bma}$ & \tiny $\bma$ & \tiny $C_{L,\bma}$ & \tiny $C_{\Delta,\bma}$ & \tiny $c_{\bma}$ & \tiny $C_{\bma}$ \\ \hline
\endhead
$(1,1,1,1)$ & $\frac{25}{3}$ & $1$ & $\frac{1}{600}$ & $1.819\cdot10^{102}$ & $(1,1,1,2)$ & $\frac{100}{13}$ & $1$ & $\frac{1}{650}$ & $2.221\cdot10^{102}$ \\ \hline
$(1,1,1,3)$ & $\frac{75}{13}$ & $\frac{2}{3}$ & $\frac{1}{1300}$ & $1.257\cdot10^{103}$ & $(1,1,1,4)$ & $\frac{25}{6}$ & $1$ & $\frac{1}{1200}$ & $1.029\cdot10^{103}$ \\ \hline
$(1,1,1,5)$ & $5$ & $1$ & $\frac{1}{1000}$ & $6.52\cdot10^{102}$ & $(1,1,1,6)$ & $\frac{25}{6}$ & $\frac{2}{3}$ & $\frac{1}{1800}$ & $2.835\cdot10^{103}$ \\ \hline
$(1,1,1,7)$ & $\frac{175}{52}$ & $\frac{6}{7}$ & $\frac{3}{5200}$ & $2.579\cdot10^{103}$ & $(1,1,1,8)$ & $\frac{50}{13}$ & $1$ & $\frac{1}{1300}$ & $1.257\cdot10^{103}$ \\ \hline
$(1,1,1,9)$ & $\frac{25}{8}$ & $\frac{2}{3}$ & $\frac{1}{2400}$ & $5.818\cdot10^{103}$ & $(1,1,1,10)$ & $\frac{20}{7}$ & $1$ & $\frac{1}{1750}$ & $2.642\cdot10^{103}$ \\ \hline
$(1,1,2,2)$ & $\frac{25}{6}$ & $1$ & $\frac{1}{1200}$ & $1.029\cdot10^{103}$ & $(1,1,2,3)$ & $\frac{25}{6}$ & $\frac{2}{3}$ & $\frac{1}{1800}$ & $2.835\cdot10^{103}$ \\ \hline
$(1,1,2,4)$ & $\frac{50}{13}$ & $1$ & $\frac{1}{1300}$ & $1.257\cdot10^{103}$ & $(1,1,2,5)$ & $\frac{20}{7}$ & $1$ & $\frac{1}{1750}$ & $2.642\cdot10^{103}$ \\ \hline
$(1,1,2,6)$ & $\frac{75}{26}$ & $\frac{2}{3}$ & $\frac{1}{2600}$ & $7.107\cdot10^{103}$ & $(1,1,2,7)$ & $\frac{35}{12}$ & $\frac{6}{7}$ & $\frac{1}{2000}$ & $3.689\cdot10^{103}$ \\ \hline
$(1,1,2,8)$ & $\frac{25}{12}$ & $1$ & $\frac{1}{2400}$ & $5.818\cdot10^{103}$ & $(1,1,2,9)$ & $\frac{30}{13}$ & $\frac{2}{3}$ & $\frac{1}{3250}$ & $1.242\cdot10^{104}$ \\ \hline
$(1,1,2,10)$ & $\frac{5}{2}$ & $1$ & $\frac{1}{2000}$ & $3.689\cdot10^{103}$ & $(1,1,2,11)$ & $\frac{550}{299}$ & $\frac{10}{11}$ & $\frac{1}{2990}$ & $1.008\cdot10^{104}$ \\ \hline
$(1,1,2,12)$ & $\frac{25}{12}$ & $\frac{2}{3}$ & $\frac{1}{3600}$ & $1.604\cdot10^{104}$ & $(1,1,2,13)$ & $\frac{13}{6}$ & $\frac{12}{13}$ & $\frac{1}{2500}$ & $6.443\cdot10^{103}$ \\ \hline
$(1,1,2,14)$ & $\frac{175}{104}$ & $\frac{6}{7}$ & $\frac{3}{10400}$ & $1.459\cdot10^{104}$ & $(1,1,2,15)$ & $\frac{30}{17}$ & $\frac{2}{3}$ & $\frac{1}{4250}$ & $2.428\cdot10^{104}$ \\ \hline
$(1,1,2,16)$ & $\frac{25}{13}$ & $1$ & $\frac{1}{2600}$ & $7.107\cdot10^{103}$ & $(1,1,2,17)$ & $\frac{425}{276}$ & $\frac{16}{17}$ & $\frac{1}{3450}$ & $1.442\cdot10^{104}$ \\ \hline
$(1,1,2,18)$ & $\frac{25}{16}$ & $\frac{2}{3}$ & $\frac{1}{4800}$ & $3.292\cdot10^{104}$ & $(1,1,2,19)$ & $\frac{950}{533}$ & $\frac{18}{19}$ & $\frac{9}{26650}$ & $9.838\cdot10^{103}$ \\ \hline
$(1,1,2,20)$ & $\frac{10}{7}$ & $1$ & $\frac{1}{3500}$ & $1.495\cdot10^{104}$ & $(1,1,2,21)$ & $\frac{175}{117}$ & $\frac{4}{7}$ & $\frac{1}{5850}$ & $5.397\cdot10^{104}$ \\ \hline
$(1,1,2,22)$ & $\frac{275}{168}$ & $\frac{10}{11}$ & $\frac{1}{3360}$ & $1.35\cdot10^{104}$ & $(1,1,2,23)$ & $\frac{575}{444}$ & $\frac{22}{23}$ & $\frac{11}{44400}$ & $2.135\cdot10^{104}$ \\ \hline
$(1,1,3,4)$ & $\frac{75}{26}$ & $\frac{2}{3}$ & $\frac{1}{2600}$ & $7.107\cdot10^{103}$ & $(1,1,3,5)$ & $\frac{5}{2}$ & $\frac{2}{3}$ & $\frac{1}{3000}$ & $1.017\cdot10^{104}$ \\ \hline
$(1,1,3,6)$ & $\frac{30}{13}$ & $\frac{1}{3}$ & $\frac{1}{6500}$ & $7.023\cdot10^{104}$ & $(1,2,2,2)$ & $\frac{50}{13}$ & $1$ & $\frac{1}{1300}$ & $1.257\cdot10^{103}$ \\ \hline
$(1,2,2,3)$ & $\frac{75}{26}$ & $\frac{2}{3}$ & $\frac{1}{2600}$ & $7.107\cdot10^{103}$ & $(1,2,2,4)$ & $\frac{25}{12}$ & $1$ & $\frac{1}{2400}$ & $5.818\cdot10^{103}$ \\ \hline
$(1,2,2,5)$ & $\frac{5}{2}$ & $1$ & $\frac{1}{2000}$ & $3.689\cdot10^{103}$ & $(1,2,2,6)$ & $\frac{25}{12}$ & $\frac{2}{3}$ & $\frac{1}{3600}$ & $1.604\cdot10^{104}$ \\ \hline
$(1,2,2,7)$ & $\frac{175}{104}$ & $\frac{6}{7}$ & $\frac{3}{10400}$ & $1.459\cdot10^{104}$ & $(1,2,2,8)$ & $\frac{25}{13}$ & $1$ & $\frac{1}{2600}$ & $7.107\cdot10^{103}$ \\ \hline
$(1,2,2,9)$ & $\frac{25}{16}$ & $\frac{2}{3}$ & $\frac{1}{4800}$ & $3.292\cdot10^{104}$ & $(1,2,2,10)$ & $\frac{10}{7}$ & $1$ & $\frac{1}{3500}$ & $1.495\cdot10^{104}$ \\ \hline
$(1,2,2,11)$ & $\frac{275}{168}$ & $\frac{10}{11}$ & $\frac{1}{3360}$ & $1.35\cdot10^{104}$ & $(1,2,2,12)$ & $\frac{75}{52}$ & $\frac{2}{3}$ & $\frac{1}{5200}$ & $4.021\cdot10^{104}$ \\ \hline
$(1,2,2,13)$ & $\frac{5}{4}$ & $\frac{12}{13}$ & $\frac{3}{13000}$ & $2.549\cdot10^{104}$ & $(1,2,2,14)$ & $\frac{35}{24}$ & $\frac{6}{7}$ & $\frac{1}{4000}$ & $2.087\cdot10^{104}$ \\ \hline
$(1,2,2,15)$ & $\frac{5}{4}$ & $\frac{2}{3}$ & $\frac{1}{6000}$ & $5.75\cdot10^{104}$ & $(1,2,2,16)$ & $\frac{25}{24}$ & $1$ & $\frac{1}{4800}$ & $3.292\cdot10^{104}$ \\ \hline
$(1,2,2,17)$ & $\frac{425}{312}$ & $\frac{16}{17}$ & $\frac{1}{3900}$ & $1.959\cdot10^{104}$ & $(1,2,2,18)$ & $\frac{15}{13}$ & $\frac{2}{3}$ & $\frac{1}{6500}$ & $7.023\cdot10^{104}$ \\ \hline
$(1,2,2,19)$ & $\frac{25}{24}$ & $\frac{18}{19}$ & $\frac{3}{15200}$ & $3.768\cdot10^{104}$ & $(1,2,3,3)$ & $\frac{100}{39}$ & $\frac{3}{5}$ & $\frac{2}{2925}$ & $1.242\cdot10^{104}$ \\ \hline
$(1,2,3,4)$ & $\frac{25}{12}$ & $\frac{2}{3}$ & $\frac{1}{3600}$ & $1.604\cdot10^{104}$ & $(1,2,3,5)$ & $\frac{30}{17}$ & $\frac{2}{3}$ & $\frac{1}{4250}$ & $2.428\cdot10^{104}$ \\ \hline
$(1,2,3,6)$ & $\frac{25}{16}$ & $\frac{2}{3}$ & $\frac{1}{4800}$ & $3.292\cdot10^{104}$ & $(1,2,3,7)$ & $\frac{175}{117}$ & $\frac{4}{7}$ & $\frac{1}{5850}$ & $5.397\cdot10^{104}$ \\ \hline
$(1,2,3,8)$ & $\frac{75}{52}$ & $\frac{2}{3}$ & $\frac{1}{5200}$ & $4.021\cdot10^{104}$ & $(1,2,3,9)$ & $\frac{25}{18}$ & $\frac{2}{9}$ & $\frac{1}{16200}$ & $6.887\cdot10^{105}$ \\ \hline
$(1,2,3,10)$ & $\frac{5}{4}$ & $\frac{2}{3}$ & $\frac{1}{6000}$ & $5.75\cdot10^{104}$ & $(1,2,3,11)$ & $\frac{275}{224}$ & $\frac{20}{33}$ & $\frac{1}{6720}$ & $7.633\cdot10^{104}$ \\ \hline
$(1,2,3,12)$ & $\frac{50}{39}$ & $\frac{3}{5}$ & $\frac{1}{5850}$ & $7.023\cdot10^{104}$ & $(1,2,3,13)$ & $\frac{25}{23}$ & $\frac{24}{39}$ & $\frac{2}{7475}$ & $9.961\cdot10^{104}$ \\ \hline
$(1,2,3,14)$ & $\frac{35}{32}$ & $\frac{4}{7}$ & $\frac{1}{8000}$ & $1.181\cdot10^{105}$ & $(1,2,3,15)$ & $\frac{15}{14}$ & $\frac{2}{3}$ & $\frac{1}{7000}$ & $8.453\cdot10^{104}$ \\ \hline
$(1,2,3,16)$ & $\frac{25}{24}$ & $\frac{2}{3}$ & $\frac{1}{7200}$ & $9.07\cdot10^{104}$ & $(1,2,3,17)$ & $\frac{1275}{1339}$ & $\frac{32}{51}$ & $\frac{8}{33475}$ & $1.321\cdot10^{105}$ \\ \hline
$(1,2,3,18)$ & $\frac{25}{26}$ & $\frac{2}{9}$ & $\frac{1}{23400}$ & $1.727\cdot10^{106}$ & $(1,2,3,19)$ & $\frac{475}{528}$ & $\frac{36}{57}$ & $\frac{1}{8800}$ & $1.498\cdot10^{105}$ \\ \hline
$(1,2,3,20)$ & $\frac{15}{17}$ & $\frac{2}{3}$ & $\frac{1}{8500}$ & $1.374\cdot10^{105}$ & $(1,2,3,21)$ & $\frac{35}{36}$ & $\frac{18}{35}$ & $\frac{1}{9000}$ & $2.062\cdot10^{105}$ \\ \hline
$(1,2,3,22)$ & $\frac{275}{312}$ & $\frac{20}{33}$ & $\frac{1}{9360}$ & $1.748\cdot10^{105}$ & $(1,2,3,23)$ & $\frac{1725}{2002}$ & $\frac{44}{69}$ & $\frac{1}{9100}$ & $1.629\cdot10^{105}$ \\ \hline
$(1,2,3,24)$ & $\frac{25}{32}$ & $\frac{2}{3}$ & $\frac{1}{9600}$ & $1.862\cdot10^{105}$ & $(1,2,3,25)$ & $\frac{5}{6}$ & $\frac{2}{3}$ & $\frac{1}{9000}$ & $1.585\cdot10^{105}$ \\ \hline
$(1,2,3,26)$ & $\frac{25}{32}$ & $\frac{24}{39}$ & $\frac{1}{10400}$ & $2.275\cdot10^{105}$ & $(1,2,3,27)$ & $\frac{10}{13}$ & $\frac{2}{3}$ & $\frac{1}{9750}$ & $1.936\cdot10^{105}$ \\ \hline
$(1,2,3,28)$ & $\frac{175}{234}$ & $\frac{4}{7}$ & $\frac{1}{11700}$ & $3.053\cdot10^{105}$ & $(1,2,3,29)$ & $\frac{725}{956}$ & $\frac{56}{87}$ & $\frac{7}{71700}$ & $2.19\cdot10^{105}$ \\ \hline
$(1,2,3,30)$ & $\frac{5}{6}$ & $\frac{3}{5}$ & $\frac{1}{9000}$ & $2.062\cdot10^{105}$ & $(1,2,3,31)$ & $\frac{775}{1092}$ & $\frac{20}{31}$ & $\frac{1}{10920}$ & $2.57\cdot10^{105}$ \\ \hline
$(1,2,4,4)$ & $\frac{25}{13}$ & $1$ & $\frac{1}{2600}$ & $7.107\cdot10^{103}$ & $(1,2,4,5)$ & $\frac{10}{7}$ & $1$ & $\frac{1}{3500}$ & $1.495\cdot10^{104}$ \\ \hline
$(1,2,4,6)$ & $\frac{75}{52}$ & $\frac{2}{3}$ & $\frac{1}{5200}$ & $4.021\cdot10^{104}$ & $(1,2,4,7)$ & $\frac{35}{24}$ & $\frac{6}{7}$ & $\frac{1}{4000}$ & $2.087\cdot10^{104}$ \\ \hline
$(1,2,4,8)$ & $\frac{25}{24}$ & $1$ & $\frac{1}{4800}$ & $3.292\cdot10^{104}$ & $(1,2,4,9)$ & $\frac{15}{13}$ & $\frac{2}{3}$ & $\frac{1}{6500}$ & $7.023\cdot10^{104}$ \\ \hline
$(1,2,4,10)$ & $\frac{5}{4}$ & $1$ & $\frac{1}{4000}$ & $2.087\cdot10^{104}$ & $(1,2,4,11)$ & $\frac{275}{299}$ & $\frac{10}{11}$ & $\frac{1}{5980}$ & $5.702\cdot10^{104}$ \\ \hline
$(1,2,4,12)$ & $\frac{25}{24}$ & $\frac{2}{3}$ & $\frac{1}{7200}$ & $9.07\cdot10^{104}$ & $(1,2,4,13)$ & $\frac{13}{12}$ & $\frac{12}{13}$ & $\frac{1}{5000}$ & $3.645\cdot10^{104}$ \\ \hline
$(1,2,4,14)$ & $\frac{175}{208}$ & $\frac{6}{7}$ & $\frac{3}{20800}$ & $8.253\cdot10^{104}$ & $(1,2,4,15)$ & $\frac{15}{17}$ & $\frac{2}{3}$ & $\frac{1}{8500}$ & $1.374\cdot10^{105}$ \\ \hline
$(1,2,4,16)$ & $\frac{25}{26}$ & $1$ & $\frac{1}{5200}$ & $4.021\cdot10^{104}$ & $(1,2,4,17)$ & $\frac{425}{552}$ & $\frac{16}{17}$ & $\frac{1}{6900}$ & $8.154\cdot10^{104}$ \\ \hline
$(1,2,4,18)$ & $\frac{25}{32}$ & $\frac{2}{3}$ & $\frac{1}{9600}$ & $1.862\cdot10^{105}$ & $(1,2,4,19)$ & $\frac{475}{533}$ & $\frac{18}{19}$ & $\frac{9}{53300}$ & $5.565\cdot10^{104}$ \\ \hline
$(1,2,4,20)$ & $\frac{5}{7}$ & $1$ & $\frac{1}{7000}$ & $8.453\cdot10^{104}$ & $(1,2,4,21)$ & $\frac{175}{234}$ & $\frac{4}{7}$ & $\frac{1}{11700}$ & $3.053\cdot10^{105}$ \\ \hline
$(1,2,4,22)$ & $\frac{275}{336}$ & $\frac{10}{11}$ & $\frac{1}{6720}$ & $7.633\cdot10^{104}$ & $(1,2,4,23)$ & $\frac{575}{888}$ & $\frac{22}{23}$ & $\frac{11}{88800}$ & $1.208\cdot10^{105}$ \\ \hline
$(1,2,4,24)$ & $\frac{75}{104}$ & $\frac{2}{3}$ & $\frac{1}{10400}$ & $2.275\cdot10^{105}$ & $(1,2,4,25)$ & $\frac{10}{13}$ & $1$ & $\frac{1}{6500}$ & $7.023\cdot10^{104}$ \\ \hline
$(1,2,4,26)$ & $\frac{5}{8}$ & $\frac{12}{13}$ & $\frac{3}{26000}$ & $1.442\cdot10^{105}$ & $(1,2,4,27)$ & $\frac{25}{36}$ & $\frac{4}{9}$ & $\frac{1}{16200}$ & $6.887\cdot10^{105}$ \\ \hline
$(1,2,4,28)$ & $\frac{35}{48}$ & $\frac{6}{7}$ & $\frac{1}{8000}$ & $1.181\cdot10^{105}$ & $(1,2,4,29)$ & $\frac{725}{1287}$ & $\frac{28}{29}$ & $\frac{7}{64350}$ & $1.671\cdot10^{105}$ \\ \hline
$(1,2,4,30)$ & $\frac{5}{8}$ & $\frac{2}{3}$ & $\frac{1}{12000}$ & $3.253\cdot10^{105}$ & $(1,2,4,31)$ & $\frac{775}{1092}$ & $\frac{30}{31}$ & $\frac{1}{7280}$ & $9.324\cdot10^{104}$ \\ \hline
$(1,2,4,32)$ & $\frac{25}{48}$ & $1$ & $\frac{1}{9600}$ & $1.862\cdot10^{105}$ & $(1,2,4,33)$ & $\frac{275}{448}$ & $\frac{20}{33}$ & $\frac{1}{13440}$ & $4.318\cdot10^{105}$ \\ \hline
$(1,2,4,34)$ & $\frac{425}{624}$ & $\frac{16}{17}$ & $\frac{1}{7800}$ & $1.108\cdot10^{105}$ & $(1,2,4,35)$ & $\frac{35}{67}$ & $\frac{6}{7}$ & $\frac{3}{33500}$ & $2.717\cdot10^{105}$ \\ \hline
$(1,2,4,36)$ & $\frac{15}{26}$ & $\frac{2}{3}$ & $\frac{1}{13000}$ & $3.973\cdot10^{105}$ & $(1,2,4,37)$ & $\frac{925}{1476}$ & $\frac{36}{37}$ & $\frac{1}{8200}$ & $1.256\cdot10^{105}$ \\ \hline
$(1,2,4,38)$ & $\frac{25}{48}$ & $\frac{18}{19}$ & $\frac{3}{30400}$ & $2.132\cdot10^{105}$ & $(1,2,4,39)$ & $\frac{25}{46}$ & $\frac{8}{13}$ & $\frac{1}{14950}$ & $5.635\cdot10^{105}$ \\ \hline
$(1,2,4,40)$ & $\frac{5}{8}$ & $1$ & $\frac{1}{8000}$ & $1.181\cdot10^{105}$ & $(1,2,4,41)$ & $\frac{1025}{2106}$ & $\frac{40}{41}$ & $\frac{1}{10530}$ & $2.346\cdot10^{105}$ \\ \hline
$(1,2,4,42)$ & $\frac{35}{64}$ & $\frac{4}{7}$ & $\frac{1}{16000}$ & $6.677\cdot10^{105}$ & $(1,2,4,43)$ & $\frac{215}{372}$ & $\frac{42}{43}$ & $\frac{7}{62000}$ & $1.523\cdot10^{105}$ \\ \hline
$(1,2,4,44)$ & $\frac{275}{598}$ & $\frac{10}{11}$ & $\frac{1}{11960}$ & $3.226\cdot10^{105}$ & $(1,2,4,45)$ & $\frac{15}{28}$ & $\frac{2}{3}$ & $\frac{1}{14000}$ & $4.782\cdot10^{105}$ \\ \hline
$(1,2,4,46)$ & $\frac{115}{208}$ & $\frac{22}{23}$ & $\frac{11}{104000}$ & $1.793\cdot10^{105}$ & $(1,2,4,47)$ & $\frac{1175}{2544}$ & $\frac{46}{47}$ & $\frac{23}{254400}$ & $2.653\cdot10^{105}$ \\ \hline
$(1,2,4,48)$ & $\frac{25}{48}$ & $\frac{2}{3}$ & $\frac{1}{14400}$ & $5.131\cdot10^{105}$ & $(1,2,4,49)$ & $\frac{175}{312}$ & $\frac{6}{7}$ & $\frac{1}{10400}$ & $2.275\cdot10^{105}$ \\ \hline
$(1,2,4,50)$ & $\frac{5}{12}$ & $1$ & $\frac{1}{12000}$ & $3.253\cdot10^{105}$ & $(1,2,4,51)$ & $\frac{1275}{2678}$ & $\frac{32}{51}$ & $\frac{4}{33475}$ & $7.473\cdot10^{105}$ \\ \hline
$(1,2,4,52)$ & $\frac{13}{24}$ & $\frac{12}{13}$ & $\frac{1}{10000}$ & $2.062\cdot10^{105}$ & $(1,2,4,53)$ & $\frac{1325}{3132}$ & $\frac{52}{53}$ & $\frac{1}{156600}$ & $3.284\cdot10^{105}$ \\ \hline
$(1,2,4,54)$ & $\frac{25}{52}$ & $\frac{4}{9}$ & $\frac{1}{23400}$ & $1.727\cdot10^{106}$ & $(1,2,4,55)$ & $\frac{55}{103}$ & $\frac{10}{11}$ & $\frac{1}{10300}$ & $2.22\cdot10^{105}$ \\ \hline
$(1,2,4,56)$ & $\frac{175}{416}$ & $\frac{6}{7}$ & $\frac{3}{41600}$ & $4.669\cdot10^{105}$ & $(1,2,4,57)$ & $\frac{475}{1056}$ & $\frac{36}{57}$ & $\frac{1}{17600}$ & $8.473\cdot10^{105}$ \\ \hline
$(1,2,4,58)$ & $\frac{145}{288}$ & $\frac{28}{29}$ & $\frac{7}{72000}$ & $2.213\cdot10^{105}$ & $(1,2,4,59)$ & $\frac{1475}{3601}$ & $\frac{58}{59}$ & $\frac{29}{360100}$ & $3.543\cdot10^{105}$ \\ \hline
$(1,2,4,60)$ & $\frac{15}{34}$ & $\frac{2}{3}$ & $\frac{1}{17000}$ & $7.769\cdot10^{105}$ & $(1,2,4,61)$ & $\frac{1525}{3003}$ & $\frac{60}{61}$ & $\frac{1}{10010}$ & $2.067\cdot10^{105}$ \\ \hline
$(1,2,4,62)$ & $\frac{155}{384}$ & $\frac{30}{31}$ & $\frac{1}{12800}$ & $3.822\cdot10^{105}$ & $(1,2,4,63)$ & $\frac{7}{16}$ & $\frac{4}{7}$ & $\frac{1}{20000}$ & $1.167\cdot10^{106}$ \\ \hline
$(1,2,4,64)$ & $\frac{25}{52}$ & $1$ & $\frac{1}{10400}$ & $2.275\cdot10^{105}$ & $(1,2,4,65)$ & $\frac{65}{173}$ & $\frac{12}{13}$ & $\frac{3}{43250}$ & $5.146\cdot10^{105}$ \\ \hline
$(1,2,4,66)$ & $\frac{275}{624}$ & $\frac{20}{33}$ & $\frac{1}{18720}$ & $9.886\cdot10^{105}$ & $(1,2,4,67)$ & $\frac{1675}{3612}$ & $\frac{66}{67}$ & $\frac{11}{120400}$ & $2.585\cdot10^{105}$ \\ \hline
$(1,2,4,68)$ & $\frac{425}{1104}$ & $\frac{16}{17}$ & $\frac{1}{13800}$ & $4.613\cdot10^{105}$ & $(1,2,4,69)$ & $\frac{1725}{4004}$ & $\frac{44}{69}$ & $\frac{1}{18200}$ & $9.214\cdot10^{105}$ \\ \hline
$(1,2,4,70)$ & $\frac{35}{76}$ & $\frac{6}{7}$ & $\frac{3}{38000}$ & $3.724\cdot10^{105}$ & $(1,2,4,71)$ & $\frac{1775}{4914}$ & $\frac{70}{71}$ & $\frac{1}{14040}$ & $4.816\cdot10^{105}$ \\ \hline
$(1,2,4,72)$ & $\frac{25}{64}$ & $\frac{2}{3}$ & $\frac{1}{19200}$ & $1.054\cdot10^{106}$ & $(1,2,4,73)$ & $\frac{1825}{3984}$ & $\frac{72}{73}$ & $\frac{3}{33200}$ & $2.657\cdot10^{105}$ \\ \hline
$(1,2,4,74)$ & $\frac{37}{104}$ & $\frac{36}{37}$ & $\frac{9}{130000}$ & $5.171\cdot10^{105}$ & $(1,2,4,75)$ & $\frac{5}{12}$ & $\frac{2}{3}$ & $\frac{1}{18000}$ & $8.963\cdot10^{105}$ \\ \hline
$(1,2,4,76)$ & $\frac{475}{1066}$ & $\frac{18}{19}$ & $\frac{9}{106600}$ & $3.148\cdot10^{105}$ & $(1,2,4,77)$ & $\frac{1925}{5424}$ & $\frac{60}{77}$ & $\frac{1}{18080}$ & $9.063\cdot10^{105}$ \\ \hline
$(1,2,4,78)$ & $\frac{25}{64}$ & $\frac{24}{39}$ & $\frac{1}{20800}$ & $1.287\cdot10^{106}$ & $(1,2,4,79)$ & $\frac{1975}{4628}$ & $\frac{78}{79}$ & $\frac{3}{35600}$ & $3.163\cdot10^{105}$ \\ \hline
$(1,2,4,80)$ & $\frac{5}{14}$ & $1$ & $\frac{1}{14000}$ & $4.782\cdot10^{105}$ & $(1,2,4,81)$ & $\frac{5}{13}$ & $\frac{4}{9}$ & $\frac{1}{29250}$ & $3.017\cdot10^{106}$ \\ \hline
$(1,2,4,82)$ & $\frac{1025}{2304}$ & $\frac{40}{41}$ & $\frac{1}{11520}$ & $2.937\cdot10^{105}$ & $(1,2,4,83)$ & $\frac{2075}{6036}$ & $\frac{82}{83}$ & $\frac{41}{603600}$ & $5.422\cdot10^{105}$ \\ \hline
$(1,2,4,84)$ & $\frac{175}{468}$ & $\frac{4}{7}$ & $\frac{1}{23400}$ & $1.727\cdot10^{106}$ & $(1,2,4,85)$ & $\frac{85}{209}$ & $\frac{16}{17}$ & $\frac{4}{26125}$ & $4.021\cdot10^{105}$ \\ \hline
$(1,2,4,86)$ & $\frac{1075}{3276}$ & $\frac{42}{43}$ & $\frac{1}{15600}$ & $6.267\cdot10^{105}$ & $(1,2,4,87)$ & $\frac{725}{1912}$ & $\frac{56}{87}$ & $\frac{7}{143400}$ & $1.239\cdot10^{106}$ \\ \hline
$(1,2,4,88)$ & $\frac{275}{672}$ & $\frac{10}{11}$ & $\frac{1}{13440}$ & $4.318\cdot10^{105}$ & $(1,2,4,89)$ & $\frac{2225}{6656}$ & $\frac{88}{89}$ & $\frac{11}{166400}$ & $5.803\cdot10^{105}$ \\ \hline
$(1,2,4,90)$ & $\frac{3}{8}$ & $\frac{2}{3}$ & $\frac{1}{20000}$ & $1.167\cdot10^{106}$ & $(1,2,4,91)$ & $\frac{175}{426}$ & $\frac{72}{91}$ & $\frac{3}{46150}$ & $6.052\cdot10^{105}$ \\ \hline
$(1,2,4,92)$ & $\frac{575}{1776}$ & $\frac{22}{23}$ & $\frac{11}{177600}$ & $6.83\cdot10^{105}$ & $(1,2,4,93)$ & $\frac{775}{2184}$ & $\frac{20}{31}$ & $\frac{1}{21840}$ & $1.454\cdot10^{106}$ \\ \hline
$(1,2,4,94)$ & $\frac{1175}{2912}$ & $\frac{46}{47}$ & $\frac{23}{291200}$ & $3.719\cdot10^{105}$ & $(1,2,4,95)$ & $\frac{19}{61}$ & $\frac{18}{19}$ & $\frac{9}{152500}$ & $7.706\cdot10^{105}$ \\ \hline
$(1,2,4,96)$ & $\frac{75}{208}$ & $\frac{2}{3}$ & $\frac{1}{20800}$ & $1.287\cdot10^{106}$ & $(1,2,4,97)$ & $\frac{2425}{6024}$ & $\frac{96}{97}$ & $\frac{1}{12550}$ & $3.638\cdot10^{105}$ \\ \hline
$(1,2,4,98)$ & $\frac{175}{576}$ & $\frac{6}{7}$ & $\frac{1}{19200}$ & $1.054\cdot10^{106}$ & $(1,2,4,99)$ & $\frac{825}{2392}$ & $\frac{20}{33}$ & $\frac{1}{23920}$ & $1.825\cdot10^{106}$ \\ \hline
$(1,2,4,100)$ & $\frac{5}{13}$ & $1$ & $\frac{1}{13000}$ & $3.973\cdot10^{105}$ & $(1,2,4,101)$ & $\frac{2525}{8073}$ & $\frac{100}{101}$ & $\frac{5}{16146}$ & $6.83\cdot10^{105}$ \\ \hline
$(1,2,4,102)$ & $\frac{425}{1248}$ & $\frac{32}{51}$ & $\frac{1}{23400}$ & $1.727\cdot10^{106}$ & $(1,2,4,103)$ & $\frac{2575}{6552}$ & $\frac{102}{103}$ & $\frac{17}{218400}$ & $3.857\cdot10^{105}$ \\ \hline
$(1,2,4,104)$ & $\frac{5}{16}$ & $\frac{12}{13}$ & $\frac{3}{52000}$ & $8.156\cdot10^{105}$ & $(1,2,4,105)$ & $\frac{21}{62}$ & $\frac{4}{7}$ & $\frac{3}{77500}$ & $2.212\cdot10^{106}$ \\ \hline
$(1,2,4,106)$ & $\frac{265}{728}$ & $\frac{52}{53}$ & $\frac{1}{14000}$ & $4.782\cdot10^{105}$ & $(1,2,4,107)$ & $\frac{2675}{9084}$ & $\frac{106}{107}$ & $\frac{53}{908400}$ & $7.93\cdot10^{105}$ \\ \hline
$(1,2,4,108)$ & $\frac{25}{72}$ & $\frac{4}{9}$ & $\frac{1}{32400}$ & $3.896\cdot10^{106}$ & $(1,2,4,109)$ & $\frac{2725}{7579}$ & $\frac{108}{109}$ & $\frac{27}{378950}$ & $4.812\cdot10^{105}$ \\ \hline
$(1,2,4,110)$ & $\frac{55}{184}$ & $\frac{10}{11}$ & $\frac{1}{18400}$ & $9.469\cdot10^{105}$ & $(1,2,4,111)$ & $\frac{185}{546}$ & $\frac{24}{37}$ & $\frac{1}{22750}$ & $1.61\cdot10^{106}$ \\ \hline
$(1,2,4,112)$ & $\frac{35}{96}$ & $\frac{6}{7}$ & $\frac{1}{16000}$ & $6.677\cdot10^{105}$ & $(1,2,4,113)$ & $\frac{2825}{9744}$ & $\frac{112}{113}$ & $\frac{1}{17400}$ & $8.235\cdot10^{105}$ \\ \hline
$(1,2,4,114)$ & $\frac{475}{1456}$ & $\frac{12}{19}$ & $\frac{3}{72800}$ & $1.892\cdot10^{106}$ & $(1,2,4,115)$ & $\frac{115}{311}$ & $\frac{22}{23}$ & $\frac{11}{155500}$ & $4.899\cdot10^{105}$ \\ \hline
$(1,2,4,116)$ & $\frac{725}{2574}$ & $\frac{28}{29}$ & $\frac{7}{128700}$ & $9.451\cdot10^{105}$ & $(1,2,4,117)$ & $\frac{13}{40}$ & $\frac{24}{39}$ & $\frac{1}{25000}$ & $2.038\cdot10^{106}$ \\ \hline
$(1,2,4,118)$ & $\frac{295}{816}$ & $\frac{58}{59}$ & $\frac{29}{408000}$ & $4.841\cdot10^{105}$ & $(1,2,4,119)$ & $\frac{2975}{10452}$ & $\frac{96}{119}$ & $\frac{1}{21775}$ & $1.443\cdot10^{106}$ \\ \hline
$(1,2,4,120)$ & $\frac{5}{16}$ & $\frac{2}{3}$ & $\frac{1}{24000}$ & $1.84\cdot10^{106}$ & $(1,2,4,121)$ & $\frac{275}{793}$ & $\frac{10}{11}$ & $\frac{1}{15860}$ & $6.532\cdot10^{105}$ \\ \hline
$(1,2,4,122)$ & $\frac{305}{1056}$ & $\frac{60}{61}$ & $\frac{1}{17600}$ & $8.473\cdot10^{105}$ & $(1,2,4,123)$ & $\frac{1025}{3368}$ & $\frac{80}{123}$ & $\frac{1}{25260}$ & $2.091\cdot10^{106}$ \\ \hline
$(1,2,4,124)$ & $\frac{775}{2184}$ & $\frac{30}{31}$ & $\frac{1}{14560}$ & $5.275\cdot10^{105}$ & $(1,2,4,125)$ & $\frac{2}{7}$ & $1$ & $\frac{1}{17500}$ & $8.353\cdot10^{105}$ \\ \hline
$(1,2,4,126)$ & $\frac{525}{1664}$ & $\frac{4}{7}$ & $\frac{3}{83200}$ & $2.641\cdot10^{106}$ & $(1,2,4,127)$ & $\frac{3175}{9504}$ & $\frac{126}{127}$ & $\frac{7}{105600}$ & $5.764\cdot10^{105}$ \\ \hline
$(1,2,4,128)$ & $\frac{25}{96}$ & $1$ & $\frac{1}{19200}$ & $1.054\cdot10^{106}$ & $(1,2,4,129)$ & $\frac{1075}{3432}$ & $\frac{28}{43}$ & $\frac{7}{171600}$ & $1.94\cdot10^{106}$ \\ \hline
$(1,2,4,130)$ & $\frac{65}{192}$ & $\frac{12}{13}$ & $\frac{1}{16000}$ & $6.677\cdot10^{105}$ & $(1,2,4,131)$ & $\frac{3275}{11869}$ & $\frac{130}{131}$ & $\frac{1}{18260}$ & $9.29\cdot10^{105}$ \\ \hline
$(1,2,5,5)$ & $\frac{20}{13}$ & $1$ & $\frac{1}{3250}$ & $1.242\cdot10^{104}$ & $(1,2,5,6)$ & $\frac{5}{4}$ & $\frac{2}{3}$ & $\frac{1}{6000}$ & $5.75\cdot10^{104}$ \\ \hline
$(1,2,5,7)$ & $\frac{70}{67}$ & $\frac{6}{7}$ & $\frac{3}{16750}$ & $4.803\cdot10^{104}$ & $(1,2,5,8)$ & $\frac{5}{4}$ & $1$ & $\frac{1}{4000}$ & $2.087\cdot10^{104}$ \\ \hline
$(1,2,5,9)$ & $\frac{15}{14}$ & $\frac{2}{3}$ & $\frac{1}{7000}$ & $8.453\cdot10^{104}$ & $(1,2,5,10)$ & $\frac{5}{6}$ & $1$ & $\frac{1}{6000}$ & $5.75\cdot10^{104}$ \\ \hline
\caption{Explicit bounds for all cases}
\label{tab:bounds}\\
\end{longtable}


\begin{thebibliography}{99}
\bibitem{BanerjeeKane}S. Banerjee and B. Kane, \begin{it}Finiteness theorems for universal sums of squares of almost primes\end{it}, preprint, arxiv:2106.05107.
\bibitem{Bhargava}M. Bhargava, \begin{it}On the Conway--Schneeberger fifteeen theorem\end{it}, in Quadratic forms and their applications, Contemp. Math. \textbf{272} (2000), 27--37.
\bibitem{Blomer}V. Blomer, \begin{it}Uniform bounds for Fourier coefficients of theta-series with arithmetic applications\end{it}, Acta Arith. \textbf{114} (2004), 1--21.
\bibitem{BosmaKane}W. Bosma and B. Kane, \begin{it}The triangular theorem of eight and representations by quadratic polynomials\end{it}, Proc. Amer. Math. Soc. \textbf{141} (2013), 1473--1486.
\bibitem{BKSun} K. Bringmann and B. Kane, \begin{it}Conjectures of Sun about sums of polygonal numbers\end{it}, submitted for publication, arxiv: 2109.14793. 
\bibitem{Cauchy} A.-L. Cauchy, \begin{it}D\'emonstration du th\'eor\`em g\'en\'eral de Fermat sur les nombres polygones\end{it}, M\'em. Sci. Math. Phys. Inst. France \textbf{14} (1813--1815), 177--220; Oeuvres compl\`etes \textbf{VI} (1905), 320--353.
\bibitem{Cho}B. Cho, \begin{it} On the number of representations of integers by quadratic forms with congruence conditions\end{it}, J. Math. Anal. Appl. \textbf{462} (2018), 999--1013.
\bibitem{Conway} J. H. Conway, \begin{it} Universal quadratic forms and the fifteen theorem\end{it}, in Quadratic forms and their applications, Contemp. Math. \textbf{272} (2000), 23--26.
\bibitem{Deligne}P. Deligne, \begin{it}La conjecture de Weil I\end{it}, Inst. Hautes \'Etudes Sci. Publ. Math. \textbf{43} (1974), 273--307.
\bibitem{DiamondShurman}F. Diamond and J. Shurman, \begin{it}A first course in modular forms\end{it}, Springer, 2005.
\bibitem{Dickson} L. E. Dickson, \begin{it}Quaternary quadratic forms representing all integers\end{it}, Amer. J. Math \textbf{49} (1927), 35--56.
\bibitem{NIST} Digital Library of Mathematical Functions, National Institute of Standards and Technology, website:  http://dlmf.nist.gov/.
\bibitem{DukeTernary}W. Duke, \begin{it}On ternary quadratic forms\end{it}, J. Number Theory \textbf{110} (2005), 37--43.
\bibitem{Ju}J. Ju, \begin{it}Universal sums of generalized pentagonal numbers\end{it}, Ramanujan J. \textbf{51} (2020), 479--494.
\bibitem{JuOh}J. Ju and B.-K. Oh, \begin{it}Universal sums of generalized octagonal numbers\end{it}, J. Number Theeory, \textbf{190} (2018), 292--302. 
\bibitem{KL}B. Kane and J. Liu, \begin{it}Universal sums of $m$-gonal numbers\end{it}, Int. Math. Res. Not., to appear.
\bibitem{KimPark}B. M. Kim and D. Park, \begin{it}A finiteness theorem for universal $m$-gonal forms\end{it}, preprint.
\bibitem{Kim} D. Kim, \begin{it}Weighted sums of generalized polygonal numbers with coefficients $1$ or $2$,\end{it}, preprint, arxiv:2006.04490.
\bibitem{Klingen}H. Klingen, \begin{it}\"Uber die Werte der Dedekindischen Zetafunktionen\end{it}, Math. Ann. \textbf{145} (1961/1962), 265--272.  


\bibitem{OnoBook}K. Ono, \begin{it}The web of modularity: arithmetic of the coefficients of modular forms and $q$-series\end{it}, CMBS Regional Conference Series in Mathematics \textbf{102} (2004), American Mathematical Society, Providence, RI, USA.  
\bibitem{RamanujanQuaternary} S. Ramanujan, \begin{it}On the expression of a number in the form $ax^2+by^2+cz^2+du^2$\end{it}, Proc. Camb. Phil. Soc. \textbf{19} (1916), 11--21.
\bibitem{Schoeneberg}B. Schoeneberg, \begin{it}Elliptic modular functions: an introduction\end{it}, Springer--Verlag, 1974.
\bibitem{Shimura}G. Shimura, \begin{it}On modular forms of half integral weight\end{it}, Ann. Math. \textbf{97} (1973), 440--481.
\bibitem{ShimuraCongruence}G. Shimura, \begin{it}Inhomogeneous quadratic forms and triangular numbers\end{it}, Amer. J. Math. \textbf{126} (2004), 191--214.
\bibitem{SiegelQuadForm}C. Siegel, \begin{it}\"Uber die analytische Theorie der Quadratischen Formen\end{it}, Ann. Math. \textbf{36} (1935), 527--606.
\bibitem{SiegelQuadFormIII}C. Siegel, \begin{it}\"Uber die analytische Theorie der Quadratischen Formen III\end{it}, Ann. Math. \textbf{38} (1937), 212--291.

\bibitem{vanderBlij} F. van der Blij, \begin{it}On the theory of quadratic forms\end{it}, Ann. Math. \textbf{50} (1949), 875--883.
\bibitem{Weil}A. Weil, \begin{it}Sur la formule de Siegel dans la th\'eorie des groupes classiques\end{it}, Acta. Math. \textbf{113} (1965), 1--87.

\bibitem{Willerding}M. E. Willerding, \begin{it} Determination of all classes of positive quaternary quadratic forms which represent all (positive) integers\end{it}, Ph.D. thsis, St. Louis University, 1948.

\bibitem{Xu}F. Xu, \begin{it}Strong approximation for certain quadratic fibrations with compact fibers\end{it}, Adv. Math. \textbf{281} (2015), 279--295.
\bibitem{Yang}T. Yang, \begin{it}An explicit formula for local densities of quadratic forms\end{it}, J. Number Theory \textbf{72} (1998), 309--356. 

\end{thebibliography}
\end{document}